\documentclass[12pt,a4paper]{amsart}
\usepackage[utf8]{inputenc}
\usepackage{amsfonts,bm}
\usepackage{latexsym}
\usepackage{amssymb}
\usepackage{amsmath}
\usepackage{color}
\usepackage[dvipsnames]{xcolor}

\usepackage{tikz}
\usepackage{enumerate}
\usepackage{mathrsfs}

\usepackage{todonotes}

\definecolor{old}{rgb}{0.6.0.4,0.3}

\definecolor{newcol1}{rgb}{0.12,0.487,0.14}
\definecolor{newcol2}{rgb}{0.5,0.7,0.5}

 \definecolor{col1}{rgb}{0.7, 0.86, 0.96}
\definecolor{col2}{rgb}{0.95, 0.95, 0.65}

\definecolor{col3}{rgb}{0.85, 1.00, 0.8}

\definecolor{texted}{rgb}{0.1, 0.6, 0.3}

\usepackage[bottom]{footmisc}

\usepackage[left=1.8cm, top=2.5cm,bottom=2.5cm,right=1.8cm]{geometry}

\newtheorem{thm}{Theorem}[section]

\newtheorem{lemma}[thm]{Lemma}
\newtheorem{df}[thm]{Definition}

\newtheorem{conj}[thm]{Conjecture}

{
\theoremstyle{definition}
\newtheorem{example}[thm]{Example}
\newtheorem{rem}[thm]{Remark}

}

\def\bx{\mathbf{x}}
\def\by{\mathbf{y}}
\def\bb{\mathbf{b}}
\def\bk{\mathbf{k}}

\def\bpi{\bm{\pi}}
\def\bsig{\bm{\sigma}}

\def\bj{\mathbf{j}}

\definecolor{gurot}{RGB}{180,20,20}
\definecolor{jorot}{RGB}{220,20,20}
\definecolor{darkorange}{rgb}{1.0, 0.55, 0.0}
\definecolor{folly}{rgb}{1.0, 0.0, 0.31}

\usepackage{nomencl}
\makenomenclature

\begin{document}


\title[]{Multitype $\Lambda$-coalescents}



\author{Johnston, S. G. G., Kyprianou, A. and Rogers, T.}

\maketitle

\begin{abstract}
Consider a multitype coalescent process in which each block has a colour in $\{1,\ldots,d\}$. Individual blocks may change colour, and some number of blocks of various colours may merge to form a new block of some colour. We show that if the law of a multitype coalescent process is invariant under permutations of blocks of the same colour, has consistent Markovian projections, and has asynchronous mergers, then it is a multitype $\Lambda$-coalescent: a process in which single blocks may change colour, two blocks of like colour may merge to form a single block of that colour, or large mergers across various colours happen at rates governed by a $d$-tuple of measures on $[0,1]^d$. We go on to identify when such processes come down from infinity. Our framework generalises Pitman's celebrated classification theorem for singletype coalescent processes, and provides a unifying setting for numerous examples that have appeared in the literature, including the seed-bank model, the island model, and the coalescent structure of continuous-state branching processes.

\medspace
\end{abstract}

\section{Introduction and statements of main results}

\subsection{Motivation}
Coalescent processes follow the formation and coagulation of clusters (referred to as ``blocks"), as governed by simple probabilistic rules. Notwithstanding interesting applications in physics and chemistry, they are best known for their time-reversed form which describes the construction of genealogical trees of evolving populations. The most important questions about these processes concern their behaviour at early times when initiated with a large (or infinite) number of blocks. This setting corresponds to looking far back in time in the genealogy of a large population, and results obtained there have important consequences for applied scientists in the field of population genetics. Of particular interest is the question of whether or not a coalescent process may ``come down from infinity", meaning that it almost surely has a finite number of blocks after any positive time.

To aid the development of a robust mathematical theory of these processes issued from infinity, certain properties are required. Individual elements should behave identically: this is codified in the concept of \emph{exchangeability}, which states that the process should be invariant under permutations of block labels. The rules should be well-behaved across different scales: a \emph{consistent} process is one in which each restriction to a finite number of elements is itself a Markov chain. Multiple events should not happen simultaneously (though this last criterion can be relaxed if one is careful \cite{schweinsberg2000coalescents}).

Famously, Pitman \cite{pitman1999coalescents} proved that these requirements together uniquely specify a class of processes known as $\Lambda$-coalescents (which were incidentally discovered simultaneously by Sagitov \cite{sagitov1999general}).
These are characterised by a measure $\Lambda$ on the unit interval, with the rule that at any moment in which a restriction of the process to $\{1,\ldots,n\}$ has $b$ blocks, any $k$ of these blocks are merging to form a single new block at rate
\begin{align} \label{eq:merger rates}
\lambda_{b,k} := \int_{[0,1]} s^{k-2} (1-s)^{b-k} \Lambda( \mathrm{d}s).
\end{align}
The majority of previously-studied coalescent processes, including those of Kingman  \cite{kingman1982coalescent} and Bolthausen and Sznitman \cite{bolthausen1998ruelle}, fall into this class. The existence of this framework has enabled the discovery of powerful general results. Important examples include the establishment of a simple criterion for exactly when a $\Lambda$-coalescent is able to come down from infinity \cite{schweinsberg2000necessary}, and even its asymptotic rate of descent \cite{berestycki2010lambda}.

The $\Lambda$-coalescent processes have no restrictions on which blocks may coalesce. In genealogical terms, they describe populations assumed to occupy a geographic region which is sufficiently small and well-connected. This assumption is too simplistic to properly capture the evolution of natural populations, however, which typically have been geographically separated for periods long enough to affect the shape of their ancestral trees. To account for this it is necessary to consider models of multiple coalescing populations. A large body of work exists under the heading of \emph{metapopulation} coalescent models, almost all of which is concerned with the limit in which the number of types is large. The ``standard structured coalescent" of Notohara \cite{notohara1990coalescent} extends Kingman's coalescent to a metapopulation setting, and establishes duality with a corresponding forward population dynamical model. Another interesting example of a structured coalescent is the recently proposed seed-bank model \cite{blath2016new}, where blocks come in two types, only one of which undergoes mergers. 

Amongst this diverse body of work, the high-level questions of how the requirements of consistency and exchangeability constrain the range of possible models and behaviours have so far gone unanswered. 
In this article we answer these questions, providing a unifying setting for such processes with multiple types in the framework of \emph{multitype $\Lambda$-coalescents}, a class of coalescent processes taking values in the collection of \emph{$d$-type partitions} of a set:

\begin{df} \label{df:dpart}
Let $S$ be a countable set. A $d$-type partition of $S$ is a $d$-tuple $\bpi = (\pi_1,\ldots,\pi_d)$ of disjoint collections of subsets of $S$ such that the union $\hat{\bpi} = \pi_1 \cup \ldots \cup \pi_d$ is a partition of $S$. We call $\hat{\bpi}$ the underlying partition of $\bpi$.
\end{df}

Clearly each block $\Gamma$ contained in the underlying partition $\hat{\bpi}$ of $\bpi$ is contained in exactly one of the collections $\pi_i$ for some $i$; in this case we refer to $i$ as the colour or type of the block $\Gamma$. (We use `type' and `colour' interchangeably throughout the article.) It is easily verified that whenever $T$ is a subset of $S$, a $d$-type partition on $S$ may be projected onto $T$ to induce a $d$-type partition on $T$. Throughout the paper, when possible we take an intrinsic approach in that we are not concerned with the underlying set $S$, but rather the number of blocks of each colour. With the $d$-type partitions of a set $S$ defined, we now offer a rather broad definition of $d$-type coalescent processes.

\begin{df} \label{df:coal}
A $d$-type coalescent process is a stochastic processes $(\bpi(t))_{t \geq 0}$ taking values in the set of $d$-type partitions of a set $S$ with the property that the blocks of the underlying partition process $(\hat{\bpi}(t))_{t \geq 0}$ merge as time passes.
\end{df}
In other words, the underlying partition process of a $d$-type coalescent process is a single-type coalescent process. 

\color{black}
We now introduce the multitype $\Lambda$-coalescents. Whenever $\mathbf{j} = (j_1,\ldots,j_d)$ is a multiindex, we use the term $\mathbf{j}$ blocks to refer to a collection of blocks containing exactly $j_i$ blocks of type $i$.  We write $\mathbf{a} \leq \mathbf{b}$ if $a_i \leq b_i$ for each $1 \leq i \leq d$. We write $\mathbf{a} < \mathbf{b}$ if $\mathbf{a} \leq \mathbf{b}$ and $\mathbf{a} \neq \mathbf{b}$. \color{black}

\begin{df}[The multitype $\Lambda$-coalescents] \label{def:main}
Write $\Lambda = \left(\rho_{ii \to i}, \rho_{ j \to i}, Q_{\to i} : i,j \in \{1,\ldots,d\} \right)$, where $\rho_{ii \to i},\rho_{j \to i} \geq 0$  and  each $Q_{\to i}$ is a measure on the unit cube $[0,1]^d$ satisfying the integrability condition 
\begin{align} \label{eq:integrability}
\int_{[0,1]^d} \left( s_i^2 + \sum_{ j \neq i} s_j \right) Q_{\to i} (\mathrm{d} s) < \infty.
\end{align}
The multitype $\Lambda$-coalescent is the coalescent process taking values in the set of $d$-type partitions of a countable set $S$ and governed by the following rule for  non-zero $\mathbf{k} \leq \mathbf{b}$: \color{black} when a projection of the process onto a subset $T$ of $S$ has $\mathbf{b}$ blocks, any $\mathbf{k}$ of these blocks are merging to form a single block of type $i$ at rate
\begin{align} \label{eq:merger rates d}
\lambda_{\mathbf{b},\mathbf{k} \to i} = \sum_{j \neq i} \mathrm{1}_{ \mathbf{k} = \mathbf{e}_j } \rho_{j \to i} + \mathrm{1}_{ \mathbf{k} = 2 \mathbf{e}_i } \rho_{ii \to i} + \int_{[0,1]^d} s^{\mathbf{k}}(1 - s)^{\mathbf{b} - \mathbf{k}} Q_{ \to i} ( \mathrm{d} s),
\end{align}
where for multiindices $\mathbf{k} \leq \mathbf{b}$ and $s \in [0,1]^d$, $s^{\mathbf{k}}(1 - s)^{\mathbf{b} - \mathbf{k}} := \prod_{ j = 1}^d s_j^{k_j} ( 1 - s_j)^{ b_j - k_j }$.
\end{df}

We now anatomise the integral formula \eqref{eq:merger rates d} for the rate $\lambda_{\mathbf{b},\mathbf{k} \to i}$ at which a collection of $\mathbf{k}$ of $\mathbf{b}$
blocks merge to form a single block of type $i$:
\begin{itemize}
\item Any individual block of type $j$ changes to type $i$ at rate $\rho_{j \to i}$. These changes can be thought of as mutations.
\item Any pair of blocks of type $i$ merge to form a single block of type $i$ at rate $\rho_{ii \to i}$.
\item At rate $Q_{\to i}(\mathrm{d}s)$ a merger event of type $i$ and involvement probabilities $s = (s_1,\ldots,s_d)$ occurs. At this event, a proportion of all blocks in the process across various types merge to form a single block of type $i$ according to the following rule: independently each block of some type $j$ elects to be involved in the merger with probability $s_j$, or not be involved with probability $1 - s_j$. 
\end{itemize}

We make a brief clarification on the large merger events:
\begin{rem} \label{rem:create}
Suppose at some moment there are $\mathbf{b}$ blocks in the system and a merger event of type $i$ occurs with involvement probabilities $s = (s_1,\ldots,s_d)$. Then for nonzero $\mathbf{k}\leq \bb$, the probability that $\mathbf{k}$ blocks are involved in the merger is given by $\prod_{j = 1}^d \binom{b_i}{k_i} s_i^{k_i}(1-s_i)^{b_i-k_i}$; after this merger there are $\mathbf{b}-\mathbf{k} + \mathbf{e}_i$ blocks in the system. 

We now clarify the situation in the $\bk = 0$ setting. When there are $\mathbf{b}$ blocks in the system, and a type $i$ merger occurs with involvement probabilities $(s_1,\ldots,s_d)$, there is a $\prod_{ j = 1}^d (1 - s_j)^{b_j}$ probability that \emph{none} of the $\bb$ blocks are involved in this merger, i.e.\ the merger is empty. We emphasise that at such an event \emph{no new type $i$ block is created during this `empty' merger}, so that after an empty merger of type $i$ there are still $\mathbf{b}$ (rather than $\mathbf{b}+\mathbf{e}_i$) blocks.
\end{rem}
\color{black}

Let us take a moment to reconnect the multitype $\Lambda$-coalescents with the single-type $\Lambda$-coalescents introduced by Pitman \cite{pitman1999coalescents} and Sagitov \cite{sagitov1999general}. Indeed, every measure $\Lambda$ on $[0,1]$ may be unpacked via a unique representation
\begin{align} \label{eq:decomposition}
\Lambda( \mathrm{d} s ) = \rho \delta_0 (\mathrm{d} s ) + s^{2} Q( \mathrm{d} s),
\end{align}
where $\rho \geq 0$, $\delta_0$ is the Dirac mass at zero, and $Q( \mathrm{d} s)$ is a measure on $[0,1]$ not charging zero (i.e. $Q( \{0\} ) = 0$). This representation induces an alternative representation to \eqref{eq:merger rates}
\begin{align} \label{eq:newrep}
\lambda_{b,k} = \rho \mathrm{1}_{ k= 2 } + \int_0^1 s^k (1-s)^{b-k} Q(\mathrm{d}s).
\end{align}
The equation \eqref{eq:newrep} is simply the single-type special case of \eqref{eq:merger rates d}.

In the next section we will give our main result, which states that any $d$-type coalescent process enjoying the properties of \emph{exchangeability}, \emph{consistency} and \emph{asynchronous mergers} may be realised as a $d$-type $\Lambda$-coalescent. 
\subsection{The classification theorem and recursions on $\mathbb{Z}_{ \geq 0}^d$} \label{sec:classification}
Our main result classifies the set of coalescent processes that are \emph{exchangeable}, \emph{consistent} and \emph{asynchronous}. We now take a brief moment to outline each of these concepts. 
We give a full description of multitype exchangeability in Section \ref{sec:exch}, but let us just say here that in essence, a $d$-type coalescent process is \emph{exchangeable} if its law is invariant under permutations of blocks of the same colour. 
Turning to consistency, we noted above that every $d$-type partition on a set $S$ induces a $d$-type partition on every subset $T$ of $S$. In particular, we say a $d$-type coalescent process on $S$ is \emph{consistent} if every projection of the process onto a subset $T$ of $S$ is a Markov process in its own filtration. Finally, we say a $d$-type $\Lambda$-coalescent is \emph{asynchronous}, or more explicitly, has \emph{asynchronous mergers}, if there are (almost-surely) no instances in which two separate collections of blocks merge to form two new blocks simultaneously. In other words, asynchronous means there are no simultaneous multiple mergers.

We now state our main result:
\begin{thm} \label{thm:main}
A $d$-type coalescent process is exchangeable, consistent and has asynchronous mergers if and only if it is a $d$-type $\Lambda$-coalescent.
\end{thm}

Theorem \ref{thm:main} is a generalisation of the main result \cite[Theorem 1]{pitman1999coalescents} of Pitman's original work on single-type $\Lambda$-coalescents, and our proof begins with analogous ideas from the de-Finetti theory. Namely, Pitman makes the crucial observation that any single-type exchangable and consistent coalescent process with asynchronous mergers must have the property that the merger rates satisfy the recursion
\begin{align} \label{eq:recursion}
\lambda_{ b,k } = \lambda_{ b+1,k}+ \lambda_{b+1, k+1} \qquad \text{for } 2 \leq k \leq b,
\end{align}
thereby reducing the problem to classifying the sets of non-negative real numbers $\left( \lambda_{b,k} : 2 \leq k \leq b \right)$ satisfying the relation \eqref{eq:recursion}. Pitman transforms the variables $\lambda_{b,k}$ in such a way that they can be subjected to de Finetti theory, finding along the way that the boundary value leads to an extra degree of freedom which appears as the $\rho$ part of the $\Lambda = (\rho, Q)$ decomposition. 

As in the single-type setting, we are also able to use a consistency argument in the multitype case to reduce the proof of Theorem \ref{thm:main} to a purely algebraic problem. Outlining our approach here, recall that 
$\lambda_{\bb,\bk \to i}$ is the rate at which, when a projection of the process has $\bb$ blocks, any $\bk$ blocks are merging to form a single block of type $i$. 
In Section \ref{sec:main proof} we consider how projections of a consistent $d$-type coalescent process onto different subsets $T$ and $T'$ of $S$ witness the same merger, and use the consistency of the process to infer that the rates $\lambda_{\bb, \bk \to i}$ satisfy the multitype analogue 
\begin{align} \label{eq:recursion1b}
\lambda_{\mathbf{b},\mathbf{k} \to i} = \lambda_{\mathbf{b} + \mathbf{e}_j,\mathbf{k} \to i}  + \lambda_{\mathbf{b}+\mathbf{e}_j,\mathbf{k}+\mathbf{e}_j \to i} \qquad j \in \{1,\ldots,d\}
\end{align}
of the recursion \eqref{eq:recursion}, where $(\mathbf{e}_1,\ldots,\mathbf{e}_d)$ is the standard basis of $\mathbb{Z}_{ \geq 0}^d$.

Our principal task in proving Theorem \ref{thm:main} is then an algebraic one: namely, to characterise the set of arrays $(\mu_{\bb,\bk} : \bk \leq \bb  \in \mathbb{Z}_{ \geq 0}^d - \{\mathbf{0},\mathbf{e}_i\} )$ indexed by $\mathbb{Z}_{\geq 0}^d - \{ \mathbf{0},\mathbf{e}_i \} $ satisfying the recursion 
\begin{align} \label{eq:recursion2}
\mu_{\mathbf{b},\mathbf{k}} = \mu_{\mathbf{b} + \mathbf{e}_j,\mathbf{k}}  + \mu_{\mathbf{b}+\mathbf{e}_j,\mathbf{k}+\mathbf{e}_j} \qquad j \in \{1,\ldots,d\}.
\end{align}
As mentioned above, this task is fairly straightforward for $d=1$, and follows from a simple translation of the array and an application of de Finetti's theorem. It transpires however that the $d \geq 2$ case is significantly more delicate, with multiple degrees of freedom arising due to the geometry of the indexing set $\mathbb{Z}_{\geq 0}^d - \{\mathbf{0},\mathbf{e}_i \}$, which contains multiple `corners'. To set up our approach here, it will be of no cost for us to work with a slightly more general class of indexing sets, the cofinite \emph{upper} sets:

\begin{df}
An upper set in $\mathbb{Z}_{ \geq 0}^d$ is a subset $\Gamma$ of $\mathbb{Z}_{ \geq 0}^d$ with the property
\begin{align*}
\bx \in \Gamma, \bx \leq \by \implies \by \in \Gamma.
\end{align*}
An upper set $\Gamma$ is \emph{cofinite} if $\mathbb{Z}_{\geq 0}^d-\Gamma$ is finite. 
An element $\bx$ of an upper set $\Gamma$ is minimal if there does not exist $\by$ in $\Gamma$ for which $\by < \bx$. We write $M$ for the minimal elements of $\Gamma$.
\end{df}
The set $\mathbb{Z}_{\geq 0}^d$ is a cofinite upper set with $\{\mathbf{0}\}$ as its single minimal element. The set $\mathbb{Z}_{ \geq 0}^d - \{ \bm{0},\mathbf{e}_i \}$ is a cofinite upper set with minimal elements $\{\mathbf{e}_1,\ldots,\mathbf{e}_{i-1},2\mathbf{e}_i,\mathbf{e}_{i+1},\ldots,\mathbf{e}_d\}$. 
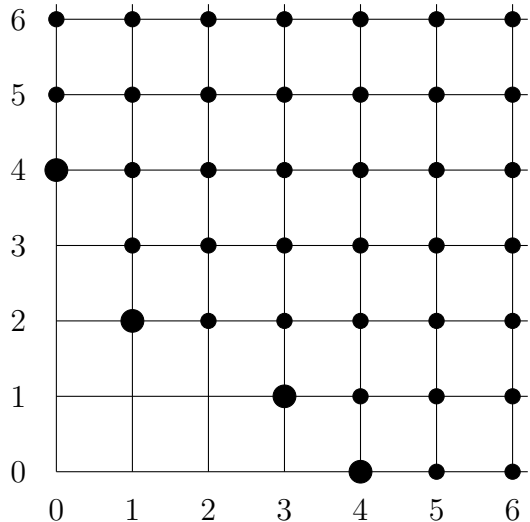
\begin{figure}[h!]
\begin{tikzpicture}
\foreach \x in {0,1,...,6}
{
\draw (\x,0) -- (\x,6.2);
\node at (\x,-0.5){\x};
}

\foreach \x in {0,1,...,6}
{
\draw (0,\x) -- (6.2,\x);
\node at (-0.5,\x){\x};
}

\foreach \x in {0,...,6}
{
\draw [fill] (\x,6) circle [radius=0.1];
}
\foreach \x in {0,...,6}
{
\draw [fill] (\x,5) circle [radius=0.1];
}
\foreach \x in {0,...,6}
{
\draw [fill] (\x,4) circle [radius=0.1];
}
\foreach \x in {1,...,6}
{
\draw [fill] (\x,3) circle [radius=0.1];
}
\foreach \x in {1,...,6}
{
\draw [fill] (\x,2) circle [radius=0.1];
}
\foreach \x in {3,...,6}
{
\draw [fill] (\x,1) circle [radius=0.1];
}
\foreach \x in {4,...,6}
{
\draw [fill] (\x,0) circle [radius=0.1];
}
\draw [fill] (0,4) circle [radius=0.15];
\draw [fill] (1,2) circle [radius=0.15];
\draw [fill] (3,1) circle [radius=0.15];
\draw [fill] (4,0) circle [radius=0.15];
\end{tikzpicture}
\caption{The dotted elements denote an upper set in $\mathbb{Z}_{\geq 0}^2$. The larger dots denote the minimal elements.}
\label{fig:covfefe}
\end{figure}

With this new definition at hand, our main algebraic result, Theorem \ref{thm:algebraic}, gives an answer to the more general problem of characterising the real-valued arrays satisfying \eqref{eq:recursion2} and doubly indexed by a cofinite upper set $\Gamma$:

\begin{thm} \label{thm:algebraic} 
If $(\mu_{\mathbf{b},\mathbf{k}} : \mathbf{k} \leq \mathbf{b} , \mathbf{k},\mathbf{b} \in \Gamma )$ is an array of non-negative real numbers satisfying the recursion \eqref{eq:recursion2} and indexed by a cofinite upper set $\Gamma$ with minimal elements $M$, then there exists a function $\rho:M \to [0,\infty)$ and a measure $J$ on $[0,1]^d$ not charging $\{0\}$ such that
\begin{align*}
\mu_{ \mathbf{b}, \mathbf{k}} = \sum_{ \mathbf{x} \in M} \mathrm{1}_{ \bk = \bx } \rho(\bx)+  \int_{[0,1]^d} s^\mathbf{k} (1 - s)^{\mathbf{b} - \mathbf{k}} J( \mathrm{d} s) \qquad \bk \leq \bb \in \Gamma.  
\end{align*}
Moreover, the measure $J$ satisfies the integrability condition
\begin{align} \label{eq:integrability2}
\int_{[0,1]^d} s^{ \bx } J(\mathrm{d} s) < \infty ~~ \qquad \text{for every $\bx$ in $\Gamma$}.
\end{align}
\end{thm}

In essence, Theorem \ref{thm:algebraic} states that the global behaviour of $\Gamma$-indexed arrays satisfying \eqref{eq:recursion2} is governed by a measure $J$ on $[0,1]^d$, though additional degrees with additional degrees of freedom arising due to the minimal elements of $\Gamma$.

Theorem \ref{thm:main} follows quickly from Theorem \ref{thm:algebraic}. Sketching the key steps here, for each $i$ the array $(\lambda_{\bb,\bk \to i } : \bk \leq \bb  \in \mathbb{Z}_{ \geq 0}^d )$ associated with the merger rates of the type $i$ mergers satisfies the recursion \eqref{eq:recursion2} and is defined for all $\bk \leq \bb$ in the cofinite upper set $\Gamma_i = \mathbb{Z}_{\geq 0}^d - \{\bm{0},\mathbf{e}_i\}$. As mentioned above, the minimal elements of $\mathbb{Z}_{ \geq 0}^d - \{ \bm{0},\mathbf{e}_i \}$ are $M_i := \{\mathbf{e}_1,\ldots,\mathbf{e}_{i-1},2\mathbf{e}_i,\mathbf{e}_{i+1},\ldots,\mathbf{e}_d\}$. It follows from Theorem \ref{thm:algebraic} that for each $i$ there is a function $\rho_i:M_i  \to [0,\infty)$ and a measure $J_i$ such that 
\begin{align*}
\lambda_{ \mathbf{b}, \mathbf{k} \to i} = \sum_{ \mathbf{x} \in M_{\mathbf{e}_i } } \mathrm{1}_{ \bk = \bx } \rho_i(\bx)+  \int_{[0,1]^d} s^\mathbf{k} (1 - s)^{\mathbf{b} - \mathbf{k}} J_i( \mathrm{d} s). 
\end{align*}
Writing $Q_{\to i}$ for $J_i$ and $\rho_{ii \to i} := \rho_i(2 \mathbf{e}_i)$ and $\rho_{j \to i } := \rho_i(\mathbf{e}_j)$, we obtain Theorem \ref{thm:main} as written.

\color{black}

In Section \ref{sec:further} we explore the algebraic structure of the rates from a different perspective, supplying probabilistic intuition for the integrability conditions \eqref{eq:integrability}, as well as exploring from a probabilistic perspective why no consistent process may have pairwise mergers at a rate $\rho_{jl \to i} > 0$ across different types (i.e. unless $j = \ell = i$).

\subsection{Coming down from infinity} \label{sec:coming down}
Take a multitype $\Lambda$-coalescent with $\Lambda = (\rho_{j \to i}, \rho_{ii \to i}, Q_{\to i} : 1 \leq i \neq j \leq d )$. Whenever $\bm \gamma$ is a $d$-type partition of a set $S$, we write $\mathbb{P}_{\bm \gamma}$ for the law of the $d$-type $\Lambda$-coalescent on $S$ starting from $\bpi(0) = \bm \gamma$. Write $N_i(t) := \# \pi_i(t)$ for the number of type $i$ blocks in the process. We say a $d$-type $\Lambda$-coalescent \emph{comes down from infinity} if 
\begin{align*}
\mathbb{P}_{\bm \gamma} \left( N_i(t) < \infty ~~ \text{for each $i = 1,\ldots,d$ and $t > 0$} \right) =1
\end{align*}
for every $d$-type partition $\bm \gamma = (\gamma_1,\ldots,\gamma_d)$ of a countable set $S$.

Shortly, we provide a characterisation of when $d$-type $\Lambda$-coalescents come down from infinity. Before this however, we review the single-type case. The first breakthrough in characterising when single-type $\Lambda$-coalescent processes come down from infinity was due to Schweinsberg \cite{schweinsberg2000necessary}, who gave a necessary and sufficient result in terms of the finiteness of a sum involving rates in the process. Bertoin and Le Gall \cite{bertoin2006stochastic} established an alternative criterion, drawing on connections with continuous-state branching processes by introducing the function
\begin{align} \label{eq:single processing}
\psi(q) := \frac{\rho}{2} q^2 +  \int_0^1 e^{ - qs} - 1 + qs ~Q( \mathrm{d}s),
\end{align}
 where we have used the $\Lambda = (\rho,Q)$ decomposition set out in \eqref{eq:decomposition}. Bertoin and Le Gall show that a $\Lambda$-coalescent comes down from infinity almost surely if and only if 
\begin{align} \label{eq:blg int}
\int_s^\infty \frac{ dq}{ \psi(q)} < \infty \qquad \text{for $s > 0$.}
\end{align}
We call the function $\psi(q)$ the processing speed, since it can be thought of as measuring the asymptotic speed of coagulation.

With a view to relating the multitype case to the well understood single-type case, given a multitype $\Lambda$-coalescent with $\Lambda = (\rho_{j \to i},\rho_{ii \to i}, Q_{\to i})$ in Section \ref{sec:coming down} we consider various projections of the $d$-type coalescent obtained by only considering the dynamics associated with blocks of a certain type. In particular, we define the type $i$ processing speed of the multitype $\Lambda$-coalescent by 
\begin{align} \label{eq:proj processing}
\psi_i(q) := \frac{ \rho_{ii \to i}}{2} q_i^2 + \int_{[0,1]^d} e^{ - q s_i} - 1 + q s_i ~Q_{ \to i }( \mathrm{d} s).
\end{align}
We highlight that the integrand in \eqref{eq:proj processing} depends only on the $i^{\text{th}}$ component $s_i$ of $s \in [0,1]^d$. We also note that since $e^{ - q s_i} - 1 + q s_i \leq C_q s_i^2$, for some constant $C_q$ depending on $q$, \eqref{eq:integrability} guarantees that $\psi_i(q)$ is finite.

The main result of this section characterises when $d$-type $\Lambda$-coalescents come down from infinity.

\begin{thm} \label{thm:coming down}
Let $(\pi(t))_{t \geq 0}$ be a $d$-type $\Lambda$-coalescent with $\Lambda = (\rho_{j \to i}, \rho_{ii \to i}, Q_{ \to i} )$, and let $\psi_1,\ldots,\psi_d$ given as in \eqref{eq:proj processing}. Then $(\pi(t))_{t \geq 0}$ comes down from infinity if and only if
\begin{align*}
\int_s^\infty \frac{ \mathrm{d} q }{ \psi_i(q) } < \infty \qquad \text{for each $i = 1,\ldots,d$ and every $s > 0$}.
\end{align*}
\end{thm}

The proof of Theorem \ref{thm:coming down} is given in Section \ref{sec:coming down}. The `only if' part of the proof is fairly straightforward, and follows from a coupling argument relating the multitype $\Lambda$-coalescent to a single-type $\Lambda$-coalescent with killing, whereupon we appeal to the Bertoin and Le Gall's criterion \cite{bertoin2006stochastic} in the single-type case. 

The `if' direction of the proof on the other hand is more involved, since in this setting we face complications unseen in the single-type case, with the most notable difficulty being the non-monotonicity of the number of blocks of a given type. Here we find the need to appeal to convexity arguments in order to control the rate at which the total number of blocks decreases. Indeed, we have a multidimensional analogue $\Psi:\mathbb{R}^d_{\geq 0} \to [0,\infty)$ of the single-type processing speed given by 
\begin{align}  \label{eq:full processing}
\Psi(q_1,\ldots,q_d) := \sum_{ i = 1}^d \frac{\rho_{ii \to i}}{2} q_i(q_i-1) + \sum_{ i  = 1}^d  \int_{[0,1]^d}  \langle \mathbf{q} , s \rangle - 1 + \prod_{ j = 1}^d (1 - s_j)^{q_j}  ~ Q_{\to i }( \mathrm{d} s)
\end{align}
 where $\langle \mathbf{n},s \rangle = \sum_{ j = 1}^d q_j s_j$ is the inner product. We will see that the function $\Psi(n_1,\ldots,n_d)$ is precisely the expected rate at which the total number of blocks in the process decreases when there are $\mathbf{n}$ blocks. (We note that in one dimension $\Psi$ differs from Bertoin and Le Gall's function $\psi$ defined in \eqref{eq:single processing}; our formula gives the precise expected rate of decrease, where as their simpler formula characterises only the large-$q$ asymptotics of the expected rate of decrease when there are $q$ blocks. The two formulas have equivalent asymptotics for large $q$, and hence the relevant ODEs in \eqref{eq:ODE} and \eqref{eq:full ODE} below have equivalent small time behaviour.) Crucially, $\Psi$ is a convex function in $\mathbf{q} \in \mathbb{R}_{ \geq 0}^d$, so that we may use Jensen's inequality to control the number of blocks in terms of the expected number of blocks, and then appeal to an ODE comparison argument. 

While in the present article we give a conclusive verdict on the whether or not a multitype $\Lambda$-coalescent comes down from infinity, the matter of the speed at which the process comes down is far more delicate, and is as-of-yet unclear to the authors. In the single-type case, Berestycki, Berestycki and Limic \cite{berestycki2010lambda} showed that roughly speaking, when the process has a large number $n$ of blocks, the asymptotic rate at which blocks coagulate is well captured by $\psi(n)$ (with $\psi$ given in \eqref{eq:single processing}). More precisely, they show that
\begin{align*}
\lim_{t \downarrow 0} N(t)/v(t) = 1 \qquad \text{almost surely},
\end{align*}
where $N(t)$ is the number of blocks in the process at time $t$, $v(t)$ is the solution to the ordinary differential equation
\begin{align} \label{eq:ODE}
\dot{v}(t) = - \psi \left( v(t) \right) \qquad v(0) = + \infty.
\end{align}
The existence of solutions to the ordinary differential equation in \eqref{eq:ODE} is equivalent to the finiteness of the integral in \eqref{eq:blg int}.

As for the multitype case, we speculate that provided one can make sense of the initial conditions, the small-$t$ asymptotics of the block numbers $(N_1(t),\ldots,N_d(t))$ are connected to ordinary differential equations associated with the 
flow on $\mathbb{R}^d$ associated with $\Phi:\mathbb{R}_{ \geq 0}^d \to \mathbb{R}^d$, where the $i^{\text{th}}$ component of $\Phi$ is given by 
\begin{align} \label{eq:phii}
\Phi_i(q_1,\ldots,q_d) &:= \sum_{ j \neq i } (\rho_{ j \to i} q_j - \rho_{i \to j } q_i ) - \rho_{ ii \to i } \frac{q_i(q_i-1)}{2}  \\
&- \sum_{ j = 1}^d \int_{[0,1]^d} s_i q_i Q_{\to j}(\mathrm{d}s) + \int_{[0,1]^d} 1 - \prod_{ j = 1}^d (1 - s_j)^{q_j} Q_{ \to i}(\mathrm{d}s). \nonumber
\end{align}
We take a moment to explain the motivation behind the function $\Phi$. When the process has $\mathbf{n}$ blocks, the $i^{\text{th}}$ component $\Phi_i(n_1,\ldots,n_d)$ of $\Phi(n_1,\ldots,n_d)$ gives the expected change in the number of blocks of type $i$ thanks to the following contributions:
\begin{itemize}
\item For $j \neq i$, at rate $\rho_{j \to i}$, a block of type $j$ changes to type $i$. It follow that if there are $n_j$ blocks of type $i$, the average increase in blocks of type $i$ due to this effect is given by $\rho_{j \to i} n_j$. Conversely, at rate $\rho_{i \to j}$, a block of type $i$ changes to colour to type $j$, and accordingly the average \emph{decrease} in blocks of type $i$ due to this effect is given by $\rho_{i \to j} n_i$. 
\item At rate $\rho_{ii \to i} \frac{n_i(n_i-1)}{2}$, a pair of blocks of type $i$ merge to form a single block of type $i$ (decreasing the total number of type $i$ blocks by $1$).
\item For each $j$, at rate $Q_{\to j}(\mathrm{d}s)$ a merger of type $j$ with involvement probabilities $s = (s_1,\ldots,s_d)$ occurs, and since each type $i$ block has a probability $s_i$ of being involved in this event, when there are $n_i$ type $i$ blocks the expected rate of \emph{decrease} in the number of blocks of type $i$ due to such events is $\sum_{ j = 1}^d \int_{[0,1]^d} s_i n_i Q_{\to j}(\mathrm{d}s)$. 
\item Finally, we have the term $ \int_{[0,1]^d} 1 - \prod_{ j= 1}^d (1 - s_j)^{n_j} Q_{ \to i}(\mathrm{d}s)  $, which is due to 
the possible creation of a new block of type $i$ at a type $i$ merger event, c.f.\ Remark \ref{rem:create}. A type $i$ merger event with involvement probabilities $s = (s_1,\ldots,s_d)$ has probability $1 - \prod_{j=1}^d (1 - s_j)^{n_j}$ of being non-empty; at such an event a `new' block of type $i$ is created.
\end{itemize}
We recall from above that the function $\Psi:\mathbb{R}_{ \geq 0}^d \to [0,\infty)$ defined in \eqref{eq:full processing} measures the expected \emph{decrease} in the total number of blocks across all types. In particular, comparing \eqref{eq:full processing} and \eqref{eq:phii} the reader may verify that 
\begin{align*}
\Psi(q_1,\ldots,q_d) = - \sum_{ i = 1}^d \Phi_i(q_1,\ldots,q_d).
\end{align*}
Under the heuristic that the small-time asymptotics for the number of blocks are governed by a multivariate ODE based on the expected changes, 
\color{black}
we conjecture the following:
\begin{conj}
Let $\Phi:\mathbb{R}^d \to \mathbb{R}^d$ be the function with components given in \eqref{eq:phii}. Then provided the conditions of Theorem \ref{thm:coming down} hold, the ordinary differential equation
\begin{align} \label{eq:full ODE} 
\dot{v}(t) = \Phi(v(t)) ~~\qquad v_i(0) = +\infty
\end{align}
has a unique solution that captures the asymptotic speed at which the $\Lambda$-coalescent comes down from infinity. That is, starting with infinitely many blocks of each type we have
\begin{align*}
\lim_{ t \downarrow 0} N_i(t)/v_i(t) = 1 ~~~~~ \text{almost surely.}
\end{align*}

\end{conj}

Let us highlight in particular one consequence of Theorem \ref{thm:coming down}, namely that the projected processing speeds $\psi_1,\ldots,\psi_d$ given in \eqref{eq:proj processing} are independent of the rates $\rho_{j \to i}$ at which individual blocks change colour, and hence by Theorem \ref{thm:coming down} the matter of whether a multitype $\Lambda$-coalescent comes down from infinity is also independent of these rates. While the rates $\rho_{j \to i}$ have no effect on whether the process comes down from infinity, they feature in the definition of $\Phi$, and we therefore believe they have a significant effect on the speed at which the numbers of different types of blocks come down from infinity.

The remainder of the paper is structured as follows. 

\begin{itemize}
\item In Section \ref{sec:further} we explore further properties of multitype $\Lambda$-coalescents, and collect together various instances of multitype $\Lambda$-coalescents occuring in the literature.
\item In Section \ref{sec:main proof} we develop our de Finetti for intratype exchangeable arrays, and use this theory to ultimately prove our main classification result, Theorem \ref{thm:main}.
\item In Section \ref{sec:coming down} we prove Theorem \ref{thm:coming down}, concerning the characterisation of when a multitype $\Lambda$-coalescent comes down from infinity. 
\end{itemize}

\section{Examples and non-examples of $\Lambda$-coalescents} \label{sec:further}

In this section we will examine multitype $\Lambda$-coalescents in full detail, looking at various structural aspects of these processes as well as several examples. We begin in the next section by studying the specific form taken by the merger rates of the coalescent.
\subsection{The impossibility of replicator pairwise mergers}
Upon closer inspection of Theorem \ref{thm:main}, the reader will note that an exchangeable and consistent coalescent process with asynchronous mergers may not have pairwise mergers across different types. That is, no such process may exist in which the merger rates take the more general form
\begin{align} \label{eq:merger rates d 2}
\lambda_{\mathbf{b},\mathbf{k} \to i} = \sum_{j \neq i} \mathrm{1}_{ \mathbf{k} = \mathbf{e}_j } \rho_{j \to i} + \sum_{ 1 \leq j \leq \ell \leq d} \rho_{j \ell \to i} \mathrm{1}_{ \mathbf{k} = \mathbf{e}_j + \mathbf{e}_\ell } + \int_{[0,1]^d} s^{\mathbf{k}}(1 - s)^{\mathbf{b} - \mathbf{k}} Q_{ \to i} ( \mathrm{d} s),
\end{align}
unless each of the terms $\rho_{j \ell \to i} = 0$ for all $j,\ell$ not both equal to $i$. 

It transpires that any process with some $\rho_{j \ell \to i} \neq 0$ for some pair $(j,\ell)$ such that $j$ and $\ell$ are not both equal to $i$ does not have projections that are Markovian in their own filtration. To see this, consider the projection of such a $\Lambda$-coalescent onto $\{ \mathbf{n}\} = \{ \mathbf{e}_j \}$, which amounts to tracking the colour of a block that started as colour $j$ at time zero. Tracking the colour of this single block, the rate at which this block changes colour to type $i$ is clearly dependent on the number of type $k$ blocks that currently exist in the system, and hence is not Markovian in its own filtration. 

\color{black}

We would like to emphasise however that one can construct interesting exchangeable coalescent processes that are not consistent, provided there is a finite number of initial blocks:

\begin{example}[An exchangeable but non-consistent $d$-type coalescent process (not a $\Lambda$-coalescent!)]
The \emph{replicator coalescent} is the exchangeable but inconsistent Markov process taking values in the set of $d$-type partitions of  $[\mathbf{n}] := \{ (i,j) : 1 \leq i \leq  d, 1 \leq j \leq n_i \}$ and defined as follows. The process starts with $n_i$ blocks of type $i$, with each $\pi_i(0)$ being the partition of $\{i\} \times \{1,\ldots,n_i\}$ into singletons. The process is then governed by the following merger rates:
\begin{itemize}
\item Any single block of type $j$ changes type to type $i$ at rate $\rho_{j \to i}$.
\item Any pair of blocks of types $j$ and $\ell$ merge to form a single block of colour $i$ at rate $\rho_{j \ell \to i}$. 
\end{itemize}
\end{example}
The replicator coalescent is a well-defined Markov process whose law is exchangeable in that it is invariant under the action of a $d$-tuple of permutations $\sigma_1,\ldots,\sigma_d$ where $\sigma_i : \{ i \} \times \{1,\ldots,n_i \} \to \{i\} \times \{1,\ldots,n_i\}$. However, this process does not have the property that its partitions induced by projection onto subsets of $[\mathbf{n}]$ are Markov processes on their own filtration. The replicator coalescents are studied in work in preparation by the latter two authors together with Lizbeth Pe\~naloza Valesco.

Conversely, the reader may wonder whether there are non-exchangeable multi-type coalescent processes that are consistent. It is easy to construct such processes:

\begin{example}[A consistent but non-exchangeable $2$-type coalescent process (not a $\Lambda$-coalescent!)]
Let $(q_k)_{k \in \mathbb{N}}$ be any collection of positive reals, not all identical, and take the process in which we start with countably many red blocks $\{1\},\{2\},\ldots$. Suppose the red block $\{k\}$ turns blue at rate $q_k$, and then the blue blocks perform a standard Kingman coalescent amongst themselves (say). This is a two-type coalescent process on $\mathbb{N}$ with Markovian projections. However this process is not exchangeable, since there are natural numbers $k \neq j$ such that the rate at which $\{k\}$ turns blue is different to the rate at which $\{j\}$ turns blue, and hence the law of the process is not invariant under the transposition of $k$ and $j$. 

\end{example}

\color{black}
\subsection{The integrability of the merger measure}

We now turn to discussing momentarily the integrability condition \eqref{eq:integrability}. The quantity
\[ r_{ii \to i} :=  \rho_{ii \to i} + \int_{[0,1]^d} s_i^2 Q_{\to i}(\mathrm{d} s) < \infty\]  
is precisely the rate at which any two separate blocks of type $i$ merge to form a single block of type $i$. Similarly, the quantity 
\[ r_{j \to i} :=  \rho_{j \to i} + \int_{[0,1]^d} s_j Q_{\to i}( \mathrm{d} s )\]
is the rate at which a single block of type $j$ changes colour to type $i$.
Both of these rates are, by the consistency of the process, independent of the remainder of the system.  Clearly both of these rates need to be finite in order to have a well-defined Markov process with the Feller property. In the single-type case the requirement that $r_{ii \to i}$ is finite is usually formulated in terms of the finiteness of $\Lambda(\mathrm{d}s) := \rho \delta_0(\mathrm(ds) + s^2 Q(\mathrm{d}s)$ as a measure on $[0,1]$.
\color{black}

The remainder of this section is dedicated to looking at examples of multitype $\Lambda$-coalescents that have appeared in the literature. We begin in the next section with the multitype Kingman coalescent first due to Notahara \cite{notohara1990coalescent}.

\subsection{Multitype Kingman coalescents}
We refer to the collection of multitype $\Lambda$-coalescents for which $Q_{\to i} = 0$ for every $i$ as the \emph{multi-type Kingman coalescents}:
\begin{example} \label{ex:multi kingman}
The multitype Kingman coalescent is the exchangeable and consistent stochastic process $(\pi(t))_{t \geq 0}$ taking values in the set of $d$-type partitions of $\mathbb{N}_d := \{ (i,j) : 1 \leq i \leq d, j \in \mathbb{N}\}$ such that at time zero each $\pi_i(0)$ is the partition of $\{i\} \times \mathbb{N}$ into singletons, and blocks merge according to the following merger rates:
\begin{itemize}
\item Any single block of type $j$ changes type to type $i$ at rate $\rho_{j \to i}$.
\item Any pair of blocks of type $i$ merge to form a single block of colour $i$ at rate $\rho_{ii \to i}$. 
\end{itemize}
\end{example}

By Theorem \ref{thm:coming down}, the multitype Kingman coalescent comes down from infinity if and only if $\rho_{ii \to i} > 0$ for each $i =1 ,\ldots,d$. These processes were first studied by Notohara, \cite{notohara1990coalescent}, who realises this group of processes in the context of a multitype version of the Wright-Fisher process, in which there is migration between types in the process.

\subsection{The coalescent structure of continuous-state branching processes}

In work by the first author with Amaury Lambert \cite{johnston2019coalescent}, multitype $\Lambda$-coalescents arise naturally in a study of the coalescent structure of the so-called $d$-dimensional continuous-state branching processes. A $d$-dimensional continuous-state branching process is a $\mathbb{R}^d_{\geq 0}$-valued Markov process enjoying the branching property: if $Z^x$ is a copy of the process starting from $x \in \mathbb{R}_{\geq 0}^d$ and $Z^y$ is a copy of the process starting from $y \in \mathbb{R}_{\geq 0}^d$, then $Z^x + Z^y$ has the same law as $Z^{x+y}$. The law of the process is governed by a \emph{branching mechanism}, a function $\phi:\mathbb{R}^d \to \mathbb{R}^d$ whose $i^{\text{th}}$ component has a L\'evy-Khintchine representation
\begin{align} \label{eq:lkrep}
\phi_i(\lambda)  = - \sum_{ j = 1}^d \kappa_{i,j} \lambda_j + \beta_i \lambda_i^2 + \int_{\mathbb{R}_{\geq 0}^d} \left( e^{ - \langle \lambda , r \rangle } - 1 +  \lambda_i r_i \mathrm{1}_{ r_i \leq 1 } \right) \nu_i(\mathrm{d}r),
\end{align}                                                                                                                                                                                                                                                                                                                                                                                                                                                                                                                                                                                                                                                                                                                                                                                                                                                                                                                                                                                                                                                                                                                                                                                                                                                                                                                                                                                                                                                                                                                                                                                                                                                                                                                                                                                                                                                                                                                                                                                                                                                                                                                                                                                                                                                                                                                                                                                                                                                                                                                                                                                                                                                                                                                                                                                                                                                                                                                                                                                                                                                                                                                                                                                                                                                                                                                                                                                                                                                                                                                                                                                                                                                                                                                                                                                                                                                                                                                                      
where $\nu_i$ is a measure on $\mathbb{R}_{ \geq 0}^d$ satisfying certain integrability conditions and determining the rates of jumps in the process. The $d$-dimensional continuous-state branching process is best thought of as appraising the evolution of a continuous $d$-type population in continuous time, and in \cite{johnston2019coalescent} the problem is considered of sampling various individuals from this population and tracing their genealogical histories back in time. These genealogical histories give rise to a $d$-type coalescent process. In \cite{johnston2019coalescent} it is shown that these coalescent processes, while non-Markovian in general, have a local representation in terms of a multitype $\Lambda$-coalescent. More explicitly, if a sample of $\mathbf{n}$ particles is taken from the population associated with a $d$-type continuous state branching process at a small time $t$, the authors show that the probability that a group of $\mathbf{b}$ blocks in the sample are descended from an ancestor of type $i$ at time zero takes precisely the form $\lambda_{\mathbf{b},\mathbf{k} \to i}$ in \eqref{eq:merger rates d}, where the quantities $\rho_{j \to i}$, $\rho_{ii \to i}$ and $Q_{\to i}$ may all be extracted from the $i^{\text{th}}$ component of the Levy-Khintchine representation. Just after a moment in which the population has size $x = (x_1,\ldots,x_d) \in \mathbb{R}_{ \geq 0}^d$, the mergers are governed by multitype $\Lambda$-coalescent type rates with:
\begin{align*}
\rho_{ii \to i} = \beta_i/x_i, ~~~\rho_{j \to i} = \kappa_{i \to j} x_i/x_j~~~ \text{and} ~~~ Q_{\to i} (\mathrm{d}s) = T_x^{ \# } \nu_{ i}
\end{align*}
where $T_x^{\#} \nu_i$ is the pushforward of $\nu$ on $\mathbb{R}_{\geq 0}^d$ onto $[0,1]^d$ via the map $r_i \mapsto \frac{ r_i}{ x_i + r_i}$. We refer the reader to Theorem 4.2 of \cite{johnston2019coalescent} for further details.

 We would like to highlight in particular the connection drawn in \cite{johnston2019coalescent} between Notohara's multitype Kingman coalescent and the $d$-dimensional Feller diffusion, the continuous-state branching process given by solutions to the $d$-dimensional stochastic differential equation 
\begin{align} \label{feller SDE}
\mathrm{d} Z_i(t) =  \sqrt{ 2  \beta_i Z_i(t) } \mathrm{d}B_i(t)  + \sum_{ j = 1}^d \kappa_{i \to j} Z_i(t)  \mathrm{d}t. 
\end{align}
The (non-negative) solutions to this equation may be thought of as giving the size of a $d$-type population, and upon making sense of taking a sample from this population at some time, its coalescent process evolves according to the following rules when the population size is $x = (x_1,\ldots,x_d)$:
\begin{itemize}
\item Any pair of blocks of type $i$ are merging to form a single block of type $i$ at rate $2 \beta_i / x_i$.
\item A block of type $j$ is turning to type $i$ at rate $ \kappa_{j \to i} x_j/x_i$.
\end{itemize}

That is, pairwise mergers within type $i$ occur at a rate inversely proportional to the type $i$ population size at that moment, and changes of colour of individual blocks depend on the ratio of the relevant population sizes. In fact, this connection between the multitype Kingman coalescent and the $d$-dimensional Feller diffusion was first noted in the single-type case by Donnelly and Kurtz \cite{donnelly1999particle}.

\subsection{The seed bank coalescent}
Blath et al. \cite{blath2016new} study a two-type variant of the Wright-Fisher model in which there are active individuals which reproduce according to the usual Wright-Fisher dynamics, as well as a dormant population who may become active at a certain rate. Blath et al. show that the coalescent process associated with sampling particles in the process and tracing their genealogies back in time is given by the following special case of Notohara's multitype Kingman coalescent seen in Example \ref{ex:multi kingman}.

\begin{example}[The Seed bank coalescent]
The seed bank Kingman coalescent is a two-type coalescent in which we have an active ($a$) type and a dormant ($d$) type, there are no giant mergers (i.e. $Q_{\to a} = Q_{\to d} = 0$), and
\begin{align*}
\rho_{aa \to a} > 0, \rho_{a \to d} > 0, \rho_{d \to a} > 0, \rho_{dd \to d} =0.
\end{align*}
\end{example}

\subsection{Further related work}
Limic and Sturm  \cite{limic2006spatial} discuss precisely the special case of our paper where for each $i$, for a fixed measure $J$ on $[0,1]$,
\begin{align*}
J_i(\mathrm{d} s ) := \delta_0( \mathrm{d}s_1) \times \ldots \delta_0( \mathrm{d} s_{i-1} ) \times J( \mathrm{d} s_i) \times \delta_0 ( \mathrm{d} s_{i+1} ) \times \ldots \times \delta_0( \mathrm{d} s_d).
\end{align*}

That is, big mergers only happen within types, and occur with the same measure for every type. Under these dynamics Limic and Sturm go on to study the space-time asymptotics of spatial $\Lambda$-coalescents on large tori in $d \geq 3$ dimensions, foreshadowing work by Angel, Berestycki and Limic \cite{angel2012global}, which we discuss in a moment. Heuer and Sturm \cite{heuer2013spatial} study a similar set up on the two-dimensional torus. 

Angel et al.\ \cite{angel2012global} look at a class of $d$-type $\Lambda$-coalescents whose types are indexed by the latttice $\mathbb{Z}^d$. The blocks whose type corresponds to a particular point in $\mathbb{Z}^d$ merge with other blocks at the same position according to the dynamics of a single-type $\Lambda$-coalescent, though we have the additional structure that the types of the underlying blocks change according to a random walk on $\mathbb{Z}^d$. The authors find for $d \geq 3$ that if the underlying mergers are governed by Kingman coalescent dynamics at each point, then starting with $N$ particles at the origin there are $O\left( (\log^*N)^{d-2} \right)$ blocks extant in the process at positive times. (Here, $\log^*x$ is the inverse of the tower function $x^x$.) Analogous results are found in $d=2$, as well as in the setting where the Kingman mergers are replaced by a beta coalescent. As in  Limic and Sturm \cite{limic2006spatial}, these coalescent processes only experience mergers within types.

Griffiths \cite{griffiths2016multi} shows that multitype coalescent processes with mutations or pairwise mergers arises in the context of coalescent duals to multitype $\Lambda$-fleming viot processes.

Freeman \cite{freeman2015segregated} constructs an extension of the single-type $\Lambda$-coalescent to the spatial continuum.

Greven, Limic and Winter \cite{greven2005representation} study genealogical questions in the setting of interacting Moran models and Fisher-Wright diffusions.

\color{black}

\section{Proof of the classification Theorem \ref{thm:main}} \label{sec:main proof}

In this section we prove Theorem \ref{thm:algebraic}, and use this result to prove our main result, Theorem \ref{thm:main} which provides a classification of exchangeable, consistent and asynchronous $d$-type coalescent processes in terms of multitype $\Lambda$-coalescents. Before proving Theorem \ref{thm:algebraic} in Section \ref{sec:algebraicproof}, we begin in the next section by exploring further the notion of $d$-type exchangeability, which was only touched on briefly in the introduction.

\subsection{Further discussion of $d$-type exchangeability} \label{sec:exch}

Recall that in the introduction we stated that a $d$-type coalescent process $(\bpi(t))_{t \geq 0}$ is exchangeable if its law is invariant under permutations of blocks of the same colour. In this section we give a more explicit description of this property and the surrounding definitions, recalling some of the notions given in the introduction.

A \emph{(single-type) partition} $\pi$ of a set $S$ is a collection of disjoint subsets (or \emph{blocks}) of $S$ whose union is $S$. A \emph{(single-type) coalescent process} on $S$ is a stochastic process $(\pi(t))_{t \geq 0}$ taking values in the set of partitions of $S$ with the property that for all times $t < t'$, every block of $\pi(t)$ is a subset of a block of $\pi(t')$. Recalling Definition \ref{df:dpart}, a $d$-type partition $\bpi = (\pi_1,\ldots,\pi_d)$ of a set $S$ is a $d$-tuple of disjoint collections of subsets of $S$ with the property that the \emph{underlying partition} $\hat{\bpi} := \pi_1 \cup \ldots \cup \pi_d$ is a single-type partition of $S$. A \emph{$d$-type coalescent process} on $S$ is a stochastic process $(\bpi(t))_{t \geq 0}$ taking values in the set of $d$-type partitions of $S$ with the property that the underlying partition process $(\widehat{\bpi(t)})_{t \geq 0}$ is a single-type coalescent process.

Let $S$ be a countable set. A bijection $\sigma:S \to S$ on a set $S$ has compact support (we call all such maps \emph{compact bijections} for short) if there are only finitely many $s \in S$ for which $\sigma(s)$ is different to $s$. 

We now recall the usual definition of exchangeability in the single-type case. Given a bijection $\sigma:\mathbb{N} \to \mathbb{N}$, and a partition $\pi$ of $\mathbb{N}$, we may define a new partition $\sigma \pi$ of $\mathbb{N}$ by letting $\sigma \pi$ be the partition of $\mathbb{N}$ whose blocks have the form $\{ \sigma^{-1}(i) : i \in \Gamma \}$ for blocks $\Gamma$ of $\pi$. That is, for $j,j'$ in $\mathbb{N}$
\begin{align*}
&\text{$j,j'$ are in the same block of $\sigma \pi$}\\
\iff  & \text{$\sigma(j),\sigma(j')$ are in the same block of $\pi$}
\end{align*}
A Markovian single-type coalescent process taking values in the set of partitions of $\mathbb{N}$ is \emph{exchangeable} if its law is invariant under the action of compact bijections $\sigma:\mathbb{N} \to \mathbb{N}$ of the natural numbers, i.e. for all such $\sigma$ the process $\left( \sigma \pi(t) \right)_{t \geq 0}$ has the same law as $(\pi(t))_{t \geq 0}$.

We now give a more intrinsic definition of exchangeability, again in the single-type case. Consider a single-type coalescent process $(\pi(t))_{t \geq 0}$ taking values in the set of partitions of a set $S$ and suppose $\pi(0)$ is countable. Since the process is a coalescent process, for each $t \geq 0$, each block of $\pi(t)$ is a union of blocks in $\pi(0)$. Let $\sigma:\pi(0) \to \pi(0)$ be a compact bijection. For each $t \geq 0$, we may create a new partition $\sigma \pi(t)$ of $S$ by declaring, for each pair of elements $\Gamma,\Gamma'$ of $\pi(0)$ (note: $\Gamma,\Gamma'$ are subsets of $S$) 
\begin{align*}
&\text{$\Gamma,\Gamma'$ are subsets of the same block of $\sigma \pi(t)$}\\
\iff &\text{$\sigma(\Gamma),\sigma(\Gamma')$ are subsets of the same block of $\pi(t)$}.
\end{align*}
Clearly the new process $(\sigma \pi (t))_{t \geq 0}$ is a coalescent process.
We say a Markovian single-type coalescent process starting from $\sigma \pi (0)$ is \emph{exchangeable} if its law is invariant under the action of bijections $\sigma:\pi(0) \to \pi(0)$, i.e. for all such $\sigma$ the process  $\left( \sigma \pi(t) \right)_{t \geq 0}$ has the same law as $(\pi(t))_{t \geq 0}$. It is clear that this instrinsic definition of exchangeability agrees with the usual definition of exchangeability in the setting that a single-type coalescent process takes values in the set of partitions of $\mathbb{N}$ and starts from the singletons.

Turning to the $d$-type case, suppose we have a $d$-type coalescent process $(\bpi(t))_{t \geq 0}$ on a set $S$. Let $\bpi(0) = (\pi_1(0),\ldots,\pi_d(0))$ denote the initial $d$-type partition of $S$, and suppose each $\pi_i(0)$ is countable. 
Note that since $(\bpi(t))_{t \geq 0} =(\pi_1(t),\ldots,\pi_d(t))_{t \geq 0}$ is a $d$-type coalescent process, each element of $\pi_i(t)$ is a union of some elements (themselves subsets of $S$) of the initial underlying partition $\widehat{\bpi(0)}:= \pi_1(0) \cup \ldots \cup \pi_d(0)$. Let $\bsig := (\sigma_1,\ldots,\sigma_d)$ be a $d$-tuple of compact bijections $\sigma_i:\pi_i(0) \to \pi_i(0)$. For each $ t \geq 0$, we may create a new $d$-type partition $\bsig \bpi(t) = (\bsig \pi_1(t),\ldots,\bsig \pi_d(t))_{t \geq 0}$ of $S$ by declaring, for each $i,i'$ and each pair of elements $\Gamma \in \pi_i(0),\Gamma' \in \pi_{i'}(0)$ (again note, $\Gamma,\Gamma'$ are subsets of $S$)
\begin{align*}
&\text{$\Gamma,\Gamma'$ are subsets of the same block of $\sigma \pi_j(t)$}\\
\iff &\text{$\sigma_i(\Gamma),\sigma_{i'}(\Gamma')$ are subsets of the same block of $\pi_j(t)$}.
\end{align*}
Clearly the new process $(\bsig\bpi (t))_{t \geq 0}$ is a $d$-type coalescent process. 
We say a Markovian $d$-type coalescent process on $S$ is \emph{exchangeable} if $(\bsig \bpi(t))_{t \geq 0}$ has the same law as $(\bpi(t))_{t \geq 0}$ for every $d$-tuple $(\sigma_1,\ldots,\sigma_d)$ of compact bijections $\sigma_i:\bpi_i(0) \to \bpi_i(0)$. In other words, the law of the process is invariant under permutations of initial blocks of the same type. 

\subsection{$d$-type exchangeability and de Finetti's theorem} \label{sec:df}

\color{black}

Recall that Theorem \ref{thm:algebraic} is concerned with arrays $(\mu_{\mathbf{b},\mathbf{k}} : \mathbf{k} \leq \mathbf{b} , \mathbf{k},\mathbf{b} \in \mathbb{Z}_{ \geq 0}^d - B_{\bm\ell} )$ of non-negative reals satisfying the recursion
\begin{align} \label{eq:recursion3}
\mu_{\mathbf{b},\mathbf{k}} = \mu_{\mathbf{b} + \mathbf{e}_j,\mathbf{k}}  + \mu_{\mathbf{b}+\mathbf{e}_j,\mathbf{k}+\mathbf{e}_j} \qquad j \in \{1,\ldots,d\}.
\end{align}
Theorem \ref{thm:algebraic} states that for any such array satisfying \eqref{eq:recursion3} there exist a collection of non-negative numbers $(\rho_x : x \in M_{\bm\ell})$ indexed by the minimal elements $M_\ell$ of $\mathbb{Z}_{\geq 0}^d - B_{\bm\ell}$ and a measure $J$ defined on the unit cube $[0,1]^d$ and not charging zero such that the $\{\rho_\bx\}$ and $J$ determine the entire array through the equation
\begin{align*}
\mu_{ \mathbf{b}, \mathbf{k}} = \sum_{ \mathbf{x} \in M_{\bm\ell} } \mathrm{1}_{ \mathbf{k} = \mathbf{x} } \rho_{\bx} + \int_{[0,1]^d} s^\mathbf{k} (1 - s)^{\mathbf{b} - \mathbf{k}} J( \mathrm{d} s). 
\end{align*}
One of our key tools in proving Theorem \ref{thm:algebraic} is a de Finetti theorem for separately exchangeable sequences. Namely, we say a collection $\mathcal{X} := (X_{i,j} : 1 \leq i \leq d, j \in \mathbb{N})$ of $d$-sequences of random variables taking values in a Borel measurable space $S$ is separately exchangeable if, for every $d$-tuple $\sigma := (\sigma_i)_{i = 1}^d$ of permutations $\sigma_i:\mathbb{N} \to \mathbb{N}$ the individually permuted sequences $\sigma \mathcal{X} := \left( X_{i, \sigma_i(j)} : 1 \leq i \leq d, j \in \mathbb{N}\right)$ have the same law as $\mathcal{X}$. 
Let $\mathcal{M}_1(S)$ denote the set of probability measures on $S$ endowed with the $\sigma$-field generated by projection maps $\mu \mapsto \mu(B)$ for $B$ Borel, and let $\mathcal{M}_1(S)^d$ denote the associated $d$-fold product space. The de Finetti theorem for separately exchangeable sequences states that if $\mathcal{X}$ is separately exchangeable under a probability measure $P$, then there is a probability measure $\Pi$ on $\mathcal{M}_1(S)^d$ such that 
\begin{align} \label{eq:separate dF}
\mathbb{P} \left( \bigcap_{ i = 1}^d \bigcap_{ j = 1}^{n_i} \{ X_{i,j} \in \mathrm{d}u_{i,j} \} \right) = \int_{ \mathcal{M}_1(S)^d} \Pi( \mathrm{d} \nu ) \prod_{ i =1 }^d \prod_{ j = 1}^{n_i} \nu_i( \mathrm{d} u_{i,j} ).
\end{align}
See Section 1.1 of Kallenberg \cite{kallenberg2006probabilistic} for details. We warn the reader that at times the term \emph{separate exchangeability} may refer to a different concept in the context of $d$-dimensional arrays, as in Chapter 7 of \cite{kallenberg2006probabilistic}. 

In the special case where the random variables $X_{i,j}$ in a separately exchangeable random sequence $\mathcal{X} := (X_{i,j} : 1 \leq i \leq d, j \in \mathbb{N} )$ take values in $\{0,1\}$, each measure $\nu_i$ occuring in \eqref{eq:separate dF} takes the form $s_i \delta_0 + (1 - s_i) \delta_1$. In particular, in this setting the de Finetti theorem for separately exchangeable sequences reads as saying there exists a probability measure $\Pi$ on $[0,1]^d$ such that
\begin{align} \label{eq:seq rep}
\mathbb{P} \left(  \bigcap_{ i = 1}^d \left\{ X_{i,1} = 1,\ldots,X_{i,k_i} = 1 , X_{i,k_i + 1 } = 0,\ldots,X_{i,b_i} = 0 \right\}  \right) = \int_{[0,1]^d} \Pi( \mathrm{d} s ) s^\bk (1-s)^{\bb - \bk} .
\end{align}
From this de Finetti theory for separately exchangeable sequences, we now extract the following key algebraic fact, which the reader will note is an analogue of Theorem \ref{thm:algebraic} in the setting in which $\ell = 0 \in \mathbb{Z}_{\geq 0}^d$, i.e. the box $B_\ell$ consists just of the origin. \color{black}

\begin{lemma} \label{lem:newlem}
Let $(\mu_{\bb,\bk} : \bk \leq \bb : \bk,\bb \in \mathbb{Z}_{ \geq 0}^d )$ be an array of non-negative reals satisfying the recursion \eqref{eq:recursion3}. Then there exists a constant $\rho_{\mathbf{0}}$ and a finite measure $J$ on $[0,1]^d$ not charging zero such that
\begin{align*}
\mu_{ \mathbf{b}, \mathbf{k}} = \mathrm{1}_{ \bk = \mathbf{0}} \rho_{\mathbf{0}} + \int_{[0,1]^d} s^\mathbf{k} (1 - s)^{\mathbf{b} - \mathbf{k}} J( \mathrm{d} s). 
\end{align*}
\end{lemma}
\begin{proof}

Let $(\mu_{\mathbf{b},\mathbf{k}} : \mathbf{k} \leq \mathbf{b} \in \mathbb{Z}_{ \geq 0 }^d)$ be an array of non-negative reals satisfying \eqref{eq:recursion3}. By dividing through by $\mu_{ \mathbf{0},\mathbf{0}}$, we may assume without loss of generality for the remainder of the proof that $\mu_{\mathbf{0},\mathbf{0}}=1$. We note now as a consequence of \eqref{eq:recursion3} that we have $0 \leq \mu_{\mathbf{b},\mathbf{k}} \leq 1$ for all $\mathbf{k} \leq \mathbf{b} \in \mathbb{Z}_{ \geq 0 }^d$.

Consider now the separately exchangeable random sequence $\mathcal{X} = (X_{i,j} : 1 \leq i \leq d , j \in \mathbb{N} )$ with the property that $\mathbb{P}( E_{\mathbf{b},\mathbf{k}} ) = \mu_{\mathbf{b},\mathbf{k}}$, where 
\begin{align*}
E_{\mathbf{b},\mathbf{k}} = \bigcap_{ i = 1}^d \{ X_{i,1} = 1,\ldots,X_{i,k_i} = 1, X_{i,k_i+1}=0,\ldots,X_{i,b_i} =  0 \}.
\end{align*}
To see that this is a consistent probability measure, we note that the cylinder events $E_{\mathbf{b},\mathbf{k}}$ generate the $\sigma$-algebra for $\mathcal{X}$. Writing $E_{\mathbf{b},\mathbf{k}}$ as a disjoint union
\begin{align*}
E_{\mathbf{b},\mathbf{k}} = E_{\mathbf{b}+\mathbf{e}_j,\mathbf{k}} \sqcup E_{\mathbf{b}+\mathbf{e}_k,\mathbf{k}+\mathbf{e}_j},
\end{align*} 
we see that the equation \eqref{eq:recursion3} ensures consistency, i.e. that $\mathbb{P}(E_{\mathbf{b},\mathbf{k}} ) = \mathbb{P}(E_{\mathbf{b}+\mathbf{e}_j,\mathbf{k}}) + \mathbb{P} (E_{\mathbf{b}+\mathbf{e}_k,\mathbf{k}+\mathbf{e}_j})$. 

\color{black}
In particular we may appeal to \eqref{eq:seq rep}, so that there exists a measure $\Pi$ on $[0,1]^d$ such that
\[ \mu_{\bb,\bk} = \int_{[0,1]^d} s^\bk (1 - s)^{ \bb - \bk} \Pi(\mathrm{d} s). \]
The result follows by writing $\Pi = \rho_\mathbf{0} \delta_0 + J$, where $\delta_0$ is the dirac mass at $0 \in [0,1]^d$ and $J$ is a measure on $[0,1]^d$ not charging zero. 
\end{proof}

\subsection{Proof of Theorem \ref{thm:algebraic}} \label{sec:algebraicproof}

We now prove Theorem \ref{thm:algebraic}.

Let $(\mu_{\mathbf{b},\mathbf{k}} : \mathbf{k} \leq \mathbf{b} , \mathbf{k},\mathbf{b} \in \Gamma )$ be an array of non-negative real numbers satisfying the recursion \eqref{eq:recursion2} and indexed by a cofinite upper set $\Gamma$ with minimal elements $M$. If $\Gamma = \mathbb{Z}_{ \geq 0}^d$ we are done by Lemma \ref{lem:newlem}; here, the unique minimal element of $\Gamma$ is $\mathbf{0}$. 

We now assume without loss of generality for the remainder of this section that $\Gamma$ is a \textbf{proper} cofinite upper subset of $\mathbb{Z}_{\geq 0}^d$, that is, $\mathbb{Z}_{\geq 0}^d-\Gamma$ is non-empty. The following quick lemma tells us about the minimal elements of such $\Gamma$.

\begin{lemma} \label{lem:minchar}
let $\Gamma$ be a proper cofinite upper subset of $\mathbb{Z}_{\geq 0}^d$, and let $M$ denote the minimal elements of $\Gamma$. Then for each $1 \leq i \leq d$, $M$ contains an element of the form $q_i \mathbf{e}_i$, where $q_i$ is in $\mathbb{Z}_{ >0}$.
\end{lemma}
\begin{proof}
Fix $1 \leq i \leq d$. We note that since $\Gamma$ is cofinite, $\Gamma$ contains an element of the form $p \mathbf{e}_i$, since otherwise, $\mathbb{Z}_{\geq 0}^d - \Gamma$ would contain the infinite set $\{ p \mathbf{e}_i : p \in \mathbb{Z}_{ \geq 0} \}$. Let $q_i \mathbf{e}_i$ denote the least such element. Then $q_i \mathbf{e}_i$ is a minimal element of $\mathbb{Z}_{\geq 0}^d$. Finally, $q_i > 0$, because otherwise $\Gamma$ would be all of $\mathbb{Z}_{\geq 0}^d$, and therefore not a proper subset of $\mathbb{Z}_{\geq 0}^d$.
\end{proof}

For each minimal element $\bx \in M$ we define a new array $(\mu^\bx_{\bb,\bk} : \bk \leq \bb \in \mathbb{Z}_{ \geq 0}^d  )$ by setting
\begin{align} \label{eq:shift}
\mu^\bx_{ \bb, \bk} := \mu_{ \bb + \bx , \bk + \bx } .
\end{align} 
Thanks to the fact that $\Gamma$ is an upper set, and hence contains all $\bx \leq \bk \leq \bb$, the new array $(\mu^\bx_{\bb,\bk} : \bk \leq \bb \in \mathbb{Z}_{ \geq 0}^d  )$ is indexed by all of $\mathbb{Z}_{\geq 0}$ (and not just a subset thereof). It follows that $(\mu^\bx_{\bb,\bk})$ lends itself to immediate characterisation by Lemma \ref{lem:newlem}: namely, there exists a $\rho(\bx) \geq 0$ and a finite measure $J^\bx$ on $[0,1]^d$ not charging zero such that
\begin{align} \label{eq:t array}
\mu^\bx_{ \bb, \bk } = \rho(\bx) \mathrm{1}_{ \bk = 0 } + \int_{[0,1]^d} s^\bk ( 1-s)^{\bb - \bk} J^\bx ( \mathrm{d} s). 
\end{align} 
Our next lemma gives us some information about the relationships between the measures $\{J^\bx : \bx \in M \}$ created in this construction.

\begin{lemma} \label{lem:scale props}
For minimal elements $\bx$ of $\Gamma$, let $J^\bx$ denote the measures defined in \eqref{eq:t array}. Then for $\bx, \by \in M$ we have: 
\begin{enumerate}
\item The equality
\begin{align} \label{eq:measure comp}
 s^{ \bx \vee \by - \bx} J^{\bx}(\mathrm{d}s) = s^{ \bx \vee \by - \by}  J^{\by}(\mathrm{d}s) 
\end{align}
of measures on $[0,1]^d$. Here $\bx \vee \by := (\max\{x_1,y_1\},\ldots,\max \{ x_d,y_d \} )$ denotes the supremum of $\bx$ and $\by$.

\item
The implication
\begin{align*}
x_j > 0 \implies \text{$J^\bx= 0$ on $\{s_j = 0 \}$}.
\end{align*}
\end{enumerate}
\end{lemma}

\begin{proof}
We investigate the overlap between translated arrays associated with distinct minimal elements $\bx$ and $\by$ of $M$. Consider that from \eqref{eq:shift}, for any $\bb \geq \bk \geq \bx \vee \by$ we have 
\begin{align} \label{eq:u array}
\mu_{ \bb , \bk} = \mu^\bx_{\bb - \bx, \bk - \bx } = \mu^\by_{\bb-\by,\bk-\by}. 
\end{align}
Note that if $\bx$ and $\by$ are distinct minimal elements, and $\bk \geq \bx \vee \by$, then $\bk$ is distinct from both $\bx$ and $\by$. Using \eqref{eq:t array} in \eqref{eq:u array}, we see that for $\bb \geq \bk \geq \bx \vee \by$ we have 
\begin{align} \label{eq:n array}
\int_{[0,1]^d} s^{ \bk - \bx} (1 - s)^{ \bb -\bk } J^{\bx}(\mathrm{d}s) = \int_{[0,1]^d} s^{ \bk - \by} (1 - s)^{ \bb - \bk } J^{\by}(\mathrm{d}s).
\end{align}
By setting $\bk = \bx \vee \by$ and $\bb = \bx \vee \by + \bj$ in \eqref{eq:n array}, for all $\bj \in \mathbb{Z}_{\geq 0}^d$ we have 
\begin{align} \label{eq:p array}
\int_{[0,1]^d} s^{ \bx \vee \by - \bx} (1 - s)^{ \bj } J^{\bx}(\mathrm{d}s) = \int_{[0,1]^d} s^{ \bk \vee \by - \by } (1 - s)^{ \bj } J^{\by}(\mathrm{d}s).
\end{align}
Since the moments of a measure on $[0,1]^d$ determine the measure (see e.g. \cite[Theorem 1.1.2]{shohat1943problem}) in order for \eqref{eq:p array} to hold for all $\bj$ in $\mathbb{Z}_{ \geq 0}^d$ it must be the case that we have the equality
\begin{align*} 
 s^{ \bx \vee \by - \bx} J^{\bx}(\mathrm{d}s) = s^{ \bx \vee \by - \by}  J^{\by}(\mathrm{d}s) 
\end{align*}
of measures on $[0,1]^d$, establishing the first part of the lemma.

We turn to proving the second part of the lemma. Recall from Lemma \ref{lem:minchar} that $M$ contains elements of the form $q_i\mathbf{e}_i$ (with $q_i > 0$) for each $1 \leq i \leq d$. Let $\bx$ and $q_i\mathbf{e}_i$ be distinct minimal elements. Then $q_i > x_i$, since otherwise we would have $\mathbf{q}_i \mathbf{e}_i < \bx$, contradicting the minimality of $\bx$. Now
\begin{align*}
 s^{ \bx \vee q_i\mathbf{e}_i - \bx} = s_i^{q_i-x_i} \qquad \text{and} \qquad s^{\bx \vee q_i \mathbf{e}_i - q_i\mathbf{e}_i} = \prod_{j\neq i}s_j^{x_j},
\end{align*}
so that in particular from \eqref{eq:measure comp} we have
\begin{align} \label{eq:scales}
s_i^{q_i-x_i}J^\bx(\mathrm{d}s) = \prod_{j\neq i} s_j^{x_j} J^{q_i\mathbf{e}_i}(\mathrm{d}s) \qquad \text{on $[0,1]^d$.}
\end{align}
From \eqref{eq:scales}, we see that for $j \neq i$ for which $x_j > 0$, the measure $s_i^{q_i-x_i}J^\bx(\mathrm{d}s)$ is zero on the subset $\{s_j=0\}$ of $[0,1]^d$. In particular, for all $j \neq i$ with $x_j > 0$, $J^{\bx}$ is zero on $\{s_j =0,s_i > 0\}$. It follows that for such $\bx$, $J^{\bx}$ is zero on $\cup_{i \neq j} \{ s_j = 0,s_i > 0 \} = \{ s_j = 0 \} - \{0\}$. Finally, since $J^\bx$ has no mass on $\{0\}$, it follows that for all minimal $\bx$ we have 
\begin{align*}
x_j > 0 \implies \text{$J^\bx= 0$ on $\{s_j = 0 \}$},
\end{align*}
completing the proof.
\end{proof}

We note that thanks to part (2) of Lemma \ref{lem:scale props}, $J^\bx$ has no support on $\{s_j = 0\}$ for any $j$ with $x_j > 0$, and it follows that we can define a measure $s^{-\bx}J^\bx(\mathrm{d}s)$ on $[0,1]^d - \{0\}$. Our next lemma further studies the domains and relations between the different measures $\{ J^\bx : \bx \in M \}$. 
\begin{lemma} \label{lem:Jdef}
Define $D_\bx := \cap_{j : x_j > 0 } \{s_j > 0 \}$. Then
\begin{enumerate}
\item $\bigcup_{ \bx \in M } D_\bx = [0,1]^d - \{0\}$. 
\item $s^{ - \bx} J^\bx(\mathrm{d}s)  = s^{ - \by} J^\by (\mathrm{d}s)$ on $D_\bx \cap D_\by$.
\end{enumerate}
\end{lemma}

\begin{proof}
To prove the first part of the lemma, we note from Lemma \ref{lem:minchar} that $M$ contains elements of the form $q_i\mathbf{e}_i$ for each $1 \leq i \leq d$.
For such elements we have $D_{q_i\mathbf{e}_i} = \{s_i > 0\}$. In particular,
\begin{align*}
\bigcup_{ \bx \in M } D_\bx \supseteq \bigcup_{ i = 1}^d D_{q_i\mathbf{e}_i} = \bigcup_{i=1}^d \{s_i > 0\} = [0,1]^d - \{0\}.
\end{align*}
As for the next part of the lemma, we note from part 1 of Lemma \ref{lem:scale props} we have 
\begin{align*} 
 s^{ \bx \vee \by - \bx} J^{\bx}(\mathrm{d}s) = s^{ \bx \vee \by - \by}  J^{\by}(\mathrm{d}s)  \qquad \text{on $[0,1]^d$}.
\end{align*}
In particular, for all $s \in D_\bx \cap D_\by = \{ s_j > 0 \text{ for every $j$ such that $x_j \vee y_j > 0$}\}$ we have  
\begin{align*} 
 s^{ - \bx} J^{\bx}(\mathrm{d}s) = s^{ - \by}  J^{\by}(\mathrm{d}s) ,
\end{align*}
as required.
\end{proof}

We are now ready to prove Theorem \ref{thm:algebraic}.

\begin{proof}[Proof of Theorem \ref{thm:algebraic}]
By Lemma \ref{lem:Jdef}, we can define a measure $J$ on $[0,1]^d - \{0\}$ by setting
\begin{align} \label{eq:saw3}
J(\mathrm{d}s) := s^{ - \bx} J^\bx(\mathrm{d}s) \qquad \text{on $D_\bx$}.
\end{align}
The first part of Lemma \ref{lem:Jdef} guarantees that a definition for $J$ is given on all of $[0,1]^d - \{0\}$; the latter part guarantees that $J$ is well defined, i.e.\ that the definitions given on the overlapping sets $D_\bx$ and $D_\by$ are consistent. 

With this choice of $J$ constructed from $\{J^\bx : \bx \in M\}$, and with $\{ \rho(\bx) :\bx \in M\}$ given in \eqref{eq:t array}, we define an array
\begin{align} \label{eq:lambda}
\lambda_{\bb,\bk} := \sum_{ \bx \in M} \rho(\bx) \mathrm{1}_{\bk=\bx} + \int_{[0,1]^d} s^{\bk}(1-s)^{\bb-\bk}J(\mathrm{d}s).
\end{align}
We now claim that 
\begin{align} \label{eq:lamu}
\lambda_{\bb,\bk} = \mu_{\bb,\bk} \qquad \text{for all $\bb \geq \bk \in \mathbb{Z}_{\geq 0}^d$}.
\end{align}
With a view to proving \eqref{eq:lamu} we note that for any $\bk \leq \bb$ in $\mathbb{Z}_{\geq 0}^d$, there exists some $\bx$ in $M$ such that $\bx \leq \bk$. In particular, by \eqref{eq:shift} and \eqref{eq:t array} on the one hand we have 
\begin{align} \label{eq:mu2}
\mu_{\bb,\bk} = \rho(\bx) \mathrm{1}_{\bk=\bx} + \int_{[0,1]^d} s^{ \bk - \bx} (1 - s)^{ \bb -\bk } J^{\bx}(\mathrm{d}s).
\end{align}
On the other hand, by the minimality of $\bx$, $\bk \geq \bx$ guarantees that $\bk$ is distinct from each minimal $\by \neq \bx$, so that by \eqref{eq:lambda} for $\bk \geq \bx$ we have 
\begin{align} \label{eq:lambda2}
\lambda_{\bb,\bk} :=  \rho(\bx) \mathrm{1}_{\bk=\bx} + \int_{[0,1]^d} s^{\bk}(1-s)^{\bb-\bk}J(\mathrm{d}s).
\end{align}
In light of \eqref{eq:mu2} and \eqref{eq:lambda2}, in order to prove \eqref{eq:lamu} it remains to show that for all $\bb \geq \bk \geq \bx$ we have
\begin{align} \label{eq:saw}
 \int_{[0,1]^d} s^{ \bk - \bx} (1 - s)^{ \bb -\bk } J^{\bx}(\mathrm{d}s) =  \int_{[0,1]^d} s^{\bk}(1-s)^{\bb-\bk}J(\mathrm{d}s).
\end{align}
Here we note that since $\bk \geq \bx$, the integrand $ s^{\bk}(1-s)^{\bb-\bk}$ on the right-hand-side of \eqref{eq:saw} is supported on $D_\bx := \cap_{ j : x_j > 0 } \{s_j >0 \}$. By part (2) of Lemma \ref{lem:scale props}, the measure $J^{\bx}$ is also supported on $\bx$. Thus we may reduce \eqref{eq:saw} to 
\begin{align*}
 \int_{D_\bx} s^{ \bk - \bx} (1 - s)^{ \bb -\bk } J^{\bx}(\mathrm{d}s) =  \int_{D_\bx} s^{\bk}(1-s)^{\bb-\bk}J(\mathrm{d}s).
\end{align*}
The equation \eqref{eq:lamu} now follows from using the definition \eqref{eq:saw3}. 

It follows from \eqref{eq:lamu} and \eqref{eq:lambda} that there exists a function $\rho:M \to [0,\infty)$ and a measure $J$ on $[0,1]^d$ not charging zero such that 
\begin{align*} 
\mu_{\bb,\bk} := \sum_{ \bx \in M} \rho(\bx) \mathrm{1}_{\bk=\bx} + \int_{[0,1]^d} s^{\bk}(1-s)^{\bb-\bk}J(\mathrm{d}s),
\end{align*}
which is the principal claim of Theorem \ref{thm:algebraic}. In order to complete the proof of Theorem \ref{thm:algebraic}, it remains to show that the integrability condition 
\begin{align} \label{eq:zib}
\int_{[0,1]^d} s^\bx J(\mathrm{d}s) < \infty \qquad \text{$\bx \in \Gamma$}
\end{align}
holds. Since every $\bx \in \Gamma$ is greater than or equal to some $\bx' \in M$, in which case $s^{\bx'} \leq s^{\bx}$ on $[0,1]^d$, it is sufficient to establish \eqref{eq:zib} for just $\bx$ in $M$. This follows immediately from the fact that by definition, $s^{\bx}J(\mathrm{d}s) = J^{\bx}(\mathrm{d}s)$, and the latter is a finite measure according to its construction in \eqref{eq:t array}.
\end{proof}

\color{black}
\subsection{Proof of Theorem \ref{thm:main}} \label{sec:mainproof}

We now prove our main classification result, Theorem \ref{thm:main}.
\begin{proof}[Proof of Theorem \ref{thm:main}]

First we note that by construction, one direction of Theorem \ref{thm:main} is trivial. Indeed, by construction every $d$-type $\Lambda$-coalescent is clearly exchangeable, consistent, and has asynchronous mergers. 

We turn to proving that every exchangeable, consistent and asynchronous coalescent process with these properties is a $d$-type $\Lambda$-coalescent. Without loss of generality, we let $(\bpi(t))_{t \geq 0} = (\pi^{(1)}(t),\ldots,\pi^{(d)}(t))_{t \geq 0}$ be a coalescent process taking values in the set of $d$-type partitions of the set $S = \mathbb{N}_d := \{(i,j) : 1 \leq i \leq d , j \in \mathbb{N} \}$ with these properties.. Throughout the proof, for multi-indices $\mathbf{n} = (n_1,\ldots,n_d)$ let $[\mathbf{n}]$ denote the subset of $\mathbb{N}_d$ given by 
\[ [\mathbf{n}] := \{ (i,j) : 1 \leq i \leq d, 1 \leq j \leq n_i \}. \]
By consistency, the projection of $\bpi(t)|_{[\mathbf{n}]}$ onto $[\mathbf{n}]$ is a Markov chain in its own filtration, taking values in the set of $d$-type partitions of $[\mathbf{n}]$. Let $\bm\pi$ and $\bm\tau$ be distinct $d$-type partitions of $[\mathbf{n}]$ and let 
\begin{align*}
\lambda_{\bm\gamma \to \bm\tau} := \lim_{ h \downarrow 0} \frac{1}{h}\mathbb{P} \left( \bpi(t+h)|_{[\mathbf{n}]} = \bm\tau  | \bpi(t)|_{[\mathbf{n}]} = \bm\gamma \right).
\end{align*}
Since $\bpi(t)$ is an asynchronous coalescent process, so is $\bpi(t)|_{[\mathbf{n}]}$, and hence the quantity  $\lambda_{ \bm\gamma \to \bm\tau}$ may only be non-zero when $\bm \tau$ may be obtained from $\bm \gamma$ by merging several blocks of the underlying partition of $\bm\gamma$  to form a single block of some type $i$ in $\bm\tau$, and leaving the other blocks and their types the same. 

We now make a key observation that follows from the exchangeability of the process: the rate $\lambda_{\bm\gamma \to \bm \tau}$ depends only on the types $\mathbf{b}$ and $\mathbf{k}$ of the underlying blocks of $\bm \gamma$ and $\bm \tau$, as well as the type $i$ of the newly merged block, and is in particular independent of the underlying size $\mathbf{n}$ of the projection. Let $\lambda_{ \bb, \bk \to i }$ denote the corresponding rate.

We now claim that the rates $\lambda_{ \bb, \bk \to i }$ satisfy the recursion
\begin{align} \label{eq:recursion4}
\lambda_{ \bb, \bk \to i }=  \lambda_{ \bb+\mathbf{e}_j, \bk \to i } + \lambda_{ \bb+\mathbf{e}_j, \bk+ \mathbf{e}_j \to i }.
\end{align}

To see that \eqref{eq:recursion4} holds, suppose that at some moment $t$, for a pair of multi-indices $\mathbf{m} \geq \mathbf{n}$, the smaller projection $\bpi(t)|_{[\mathbf{n}]}$ has $\mathbf{b}$ blocks and the larger projection $\bpi(t)|_{[\mathbf{m}]}$  has $\mathbf{b}+ \mathbf{e}_j$ blocks. Given any subcollection of $\mathbf{k}$ of the $\mathbf{b}$ blocks in the smaller projection, there are two possible $\to i$ merger events that can happen for the larger projection that are witnessed by the smaller projection as $\mathbf{k}$ blocks merging to a single block of type $i$. One of these events involves $\mathbf{k} + \mathbf{e}_j$ blocks merging out of the $\mathbf{b} + \mathbf{e}_j$, the other involving $\mathbf{k}$ of $\mathbf{b} + \mathbf{e}_j$ merging. The equation \eqref{eq:recursion4} follows.  

In particular, since the array $( \lambda_{ \bb, \bk \to i } )$ is indexed by $\bk \leq \bb \in \mathbb{Z}_{ \geq 0}^d - B_\ell$, we are in the setting of Theorem \ref{thm:algebraic}, so that by Theorem \ref{thm:algebraic} there constants $\rho_{ j \to i}$ and $\rho_{ii \to i}$ and measures $Q_{\to i}$ such that 
\begin{align} \label{eq:merger rates d 2}
\lambda_{\mathbf{b},\mathbf{k} \to i} = \sum_{j \neq i} \mathrm{1}_{ \mathbf{k} = \mathbf{e}_j } \rho_{j \to i} + \mathrm{1}_{ \mathbf{k} = 2 \mathbf{e}_i } \rho_{ii \to i} + \int_{[0,1]^d} s^{\mathbf{k}}(1 - s)^{\mathbf{b} - \mathbf{k}} Q_{ \to i} ( \mathrm{d} s),
\end{align}
completing the proof of Theorem \ref{thm:main}.
\end{proof}

\section{Coming down from infinity: proofs} \label{sec:coming down}

Let $\Lambda = ( \rho_{ j \to i}, \rho_{ii \to i }, Q_{ \to i} : 1 \leq i \neq j \leq d )$. In this section  we prove Theorem \ref{thm:coming down}, which states that if $\bpi(t)$ is a $\Lambda$-coalescent starting with infinitely many blocks of each colour, then $\bpi(t)$ comes down from infinity if and only if 
\begin{align*}
\int_{ s}^\infty \frac{\mathrm{d} q }{ \psi_i(q)} < \infty \qquad \text{for all $s >0$ and $i \in \{1,\ldots,d\}$}
\end{align*}  
where $\psi_i$ is defined in \eqref{eq:proj processing}. 

It is useful to consider the type $i$ \emph{projected coalescent}, the single-type $\Lambda$-coalescent with pairwise mergers at rate $\rho_{ii \to i}$, and merger measure on $[0,1]$ given by the projection $\bar{Q}_{ \to i}$ of $Q_{\to i}$ onto $(0,1]$ under the map $(s_1,\ldots,s_d) \mapsto s_i$. In other words, for Borel subsets $A$ of $(0,1]$ 
\begin{align*}
\bar{Q}_{ \to i }(A) := \int_{ [0,1]^d } \mathrm{1} \{ s_i \in A \} Q_{\to i }( \mathrm{d} s).
\end{align*}
More explicitly, the type $i$ projected coalescent is the single-type $\Lambda$-coalescent associated with the measure
\begin{align*}
\bar{\Lambda}_i(\mathrm{d}s) = \rho_{ii \to i } \delta_0( \mathrm{d}s) + s^{-2} \bar{Q}_{\to i }( \mathrm{d} s) \qquad s \in [0,1].
\end{align*}
The $\bar{Q}_{ \to i }$-coalescent may be thought of as the single-type $\Lambda$-coalescent that may be obtained from the multi-type $\Lambda$-coalescent starting with only blocks of type $i$, and `switching off' any mergers or changes of colour of type $i$ blocks into other colours.

We begin by proving the `only if' direction of the proof. 
\subsection{Proof of the `only if' direction in Theorem \ref{thm:coming down}}

Fix a $k \in \{1,\ldots,d\}$. In this section we show that if $(\bpi(t))_{t \geq 0}$ is a $\Lambda$-coalescent that comes down from infinity, then $\int_s^{\infty} \frac{ \mathrm{d}q}{\psi_k(q)} < \infty$ for each $s > 0$. 

To this end, let $\Lambda = ( \rho_{ j \to i}, \rho_{ii \to i }, Q_{ \to i} : 1 \leq i < j \leq d )$, and suppose $(\bpi(t))_{t \geq 0} = (\pi_1(t),\ldots,\pi_d(t))_{t \geq 0}$ is a $\Lambda$-coalescent that comes down from infinity. Suppose further that $(\bpi(t))_{t \geq 0}$ starts from the $d$-type partition of $\mathbb{N}$ with
\begin{align} \label{eq:kdd}
\pi_k(t) = \{ \{ n \} : n \in \mathbb{N} \}, \qquad \pi_i(0) = \varnothing \text{ for $i \neq k$}. 
\end{align}
That is, we begin with the singletons of $\mathbb{N}$, each with type $k$. 

Consider now that the $k^{\text{th}}$ component of $(\bpi(t))_{t \geq 0}$ is not a Markov process in its own filtration. Indeed, while the type $k$ blocks may merge and also change colour away from type $k$ according to Markovian dynamics, there are also new type-$k$ blocks being created at rates that depend on the number of blocks of other types in the system. It is, however, possible to construct a Markov process from the type $k$ blocks of $(\bpi(t))_{t \geq 0}$, by taking a version of the process $\pi_k(t)$ in which we do not allow `new' blocks of type $k$ to appear. We encourage the reader to take a glance at Figure \ref{fig:ring} to garner an idea before we give a precise definition.

With our multitype $\Lambda$-coalescent $(\bpi(t))_{t \geq 0}$ on $\mathbb{N}$ we can associate a function $f_t:\mathbb{N} \to \{1,\ldots,d\}$ by letting $f_t(n)$ denote the colour of the block containing $n$ at time $t$, in other words 
\begin{align*}
f_t(n) = i \iff \text{The block of $\bpi(t)$ containing $n$ has type $i$}.
\end{align*}
Given a block $\Gamma$ of $\pi_k(t)$, we define the subset $\mathrm{Always}_k(\Gamma,t)$ of $\Gamma$ to be the set of elements who were always contained in a block of type $k$ leading up to time $t$, i.e. 
\begin{align*}
\mathrm{Always}_k(\Gamma,t) := \{ n \in \Gamma :  f_s(n) = k ~\forall s \in [0,t] \} \qquad \Gamma \in \pi_k(t).
\end{align*}
We now define $\mathring{\pi}_k(t)$ to be the set of non-empty blocks that arise in this way, i.e.
\begin{align*}
\mathring{\pi}_k(t) := \left\{ \mathrm{Always}_k(\Gamma,t) \text{ nonempty} : \Gamma \in \pi_k(t) \right\} .
\end{align*}
$\mathring{\pi}_k(t)$ is a partition of a subset of $\mathbb{N}$, and the cardinality of the collection $\mathring{\pi}_k(t)$ is at most that of $\pi_k(t)$. 
Clearly $\mathring{\pi}_k(0) = \pi_k(0)$. Since $(\bpi(t))_{t \geq 0}$ comes down from infinity, we have $\mathbb{P} ( \# \pi_k(t) < \infty~ \forall t = 1)$, and as a result, $\mathbb{P} ( \# \mathring{\pi}_k(t) < \infty ~ \forall t = 1)$.

When $(\kappa(t))_{t \geq 0}$ is a stochastic process such that each $\kappa(t)$ is equal to a collection of subsets of $\mathbb{N}$, we write
\begin{align} \label{eq:restriction}
\kappa(t)|_n := \{ \Gamma \cap [n] \text{ nonempty} : \Gamma \in \kappa(t) \} ,
\end{align}
for its projection onto $\{1,\ldots,n\}$. Figure \ref{fig:ring} depicts a projection of a multitype $\Lambda$-coalescent with starting configuration \eqref{eq:kdd} onto $\{1,\ldots,n\}$, as well as the projection of the associated process $(\mathring{\pi}_k(t))_{t \geq 0}$. 
\color{black}

\begin{figure}[h!]
\begin{tikzpicture}
\draw[dotted, gray, line width=0.2mm] (0,10.2) -- (0,0.2);
\draw[dotted, gray, line width=0.2mm] (3.3,10.2) -- (3.3,0.2);
\foreach \x in {1,2,...,10}
{
\draw[line width=1.2mm] (0,\x) -- (0.6,\x);
\node at (-0.5,\x){\x};
}
\draw[line width=1.0mm] (0.55,10) -- (0.55,8);

\draw[line width=1.2mm] (0.6,8.8) -- (1.7,8.8);

\draw[line width=1.2mm] (0.6,7) -- (1.0,7);
\draw[line width=1.2mm] (0.6,6) -- (1.0,6);
\draw[line width=1.0mm, newcol1] (0.95,5.94) -- (0.95,7.06);
\draw[line width=1.2mm, newcol1] (1.0,6.5) -- (1.3,6.5);
\draw[line width=1.2mm] (0.6,5) -- (0.8,5);
\draw[line width=1.2mm, newcol1] (0.8,5) -- (1.3,5);
\draw[line width=1.0mm, newcol1] (1.25,5) -- (1.25,6.5);
\draw[line width=1.2mm, newcol1] (1.3,5.75) -- (1.7,5.75);

\node at (3.3,7.27){$\{5,6,7,8,9,10\}$};
\draw[line width=1.0mm] (1.65,5.69) -- (1.65,8.86);
\draw[line width=1.2mm] (1.7,6.77) -- (3.3,6.77);

\draw[line width=1.2mm] (0.6,4)--(0.72,4);
\draw[line width=1.2mm, newcol2] (0.72,4) -- (1.65,4);
\draw[line width=1.2mm, newcol2] (1.6,2.94) -- (1.6,4.06);
\draw[line width=1.2mm, newcol2] (1.65,3.5) -- (3.3,3.5);
\draw[line width=1.2mm, newcol1] (1.2,3) -- (1.65,3);
\draw[line width=1.2mm] (0.6,3) -- (1.2,3);

\node at (3.3,4){\textcolor{newcol2}{$\{3,4\}$}};

\draw[line width=1.2mm] (0.55,0.94) -- (0.55,2.06);
\draw[line width=1.2mm] (0.6,1.5) -- (3.3,1.5);

\node at (3.3,2){$\{1,2\}$};

\node at (0,-0.1){$t=0$};

\node at (3.3,-0.1){$t=t_0$};


\draw[dotted, gray, line width=0.2mm] (7,10.2) -- (7,0.2);
\draw[dotted, gray, line width=0.2mm] (10.3,10.2) -- (10.3,0.2);
\foreach \x in {1,2,...,10}
{
\draw[line width=1.2mm] (7,\x) -- (7.6,\x);
\node at (6.5,\x){\x};
}
\draw[line width=1.0mm] (7.55,10) -- (7.55,8);

\draw[line width=1.2mm] (7.6,8.8) -- (10.3,8.8);

\draw[line width=1.2mm] (7.6,7) -- (8.0,7);
\draw[line width=1.2mm] (7.6,6) -- (8.0,6);
\draw[line width=1.2mm] (7.6,5) -- (7.8,5);

\draw[line width=1.0, gray] (8.0,5.94) -- (8.0,7.06);

\node at (10.3,9.3){$\{8,9,10\}$};

\draw[line width=1.2mm] (7.6,4)--(7.72,4);
\draw[line width=1.2mm] (7.6,3) -- (7.2,3);


\draw[line width=1.2mm] (7.55,0.94) -- (7.55,2.06);
\draw[line width=1.2mm] (7.6,1.5) -- (10.3,1.5);

\node at (10.3,2){$\{1,2\}$};

\node at (7,-0.1){$t=0$};

\node at (10.3,-0.1){$t=t_0$};

\end{tikzpicture}
\caption{On the left we have a projection $(\bpi(t))_{t \geq 0}$ onto $[10] := \{1,\ldots,10\}$. The type $k$ blocks are coloured in black; blocks of other types are coloured in lighter shades. At time $t_0$ we have $\pi_k(t_0)|_{10} = \{\{1,2\},\{5,6,7,8,9,10\} \}$. On the right we have the projection of the associated process $(\mathring{\pi}_k(t))_{t \geq 0}$ onto $[10]$. Here $\mathring{\pi}_k(t_0)|_{10} = \{ \{1,2\}, \{8,9,10\} \}$.}
\label{fig:ring}
\end{figure}
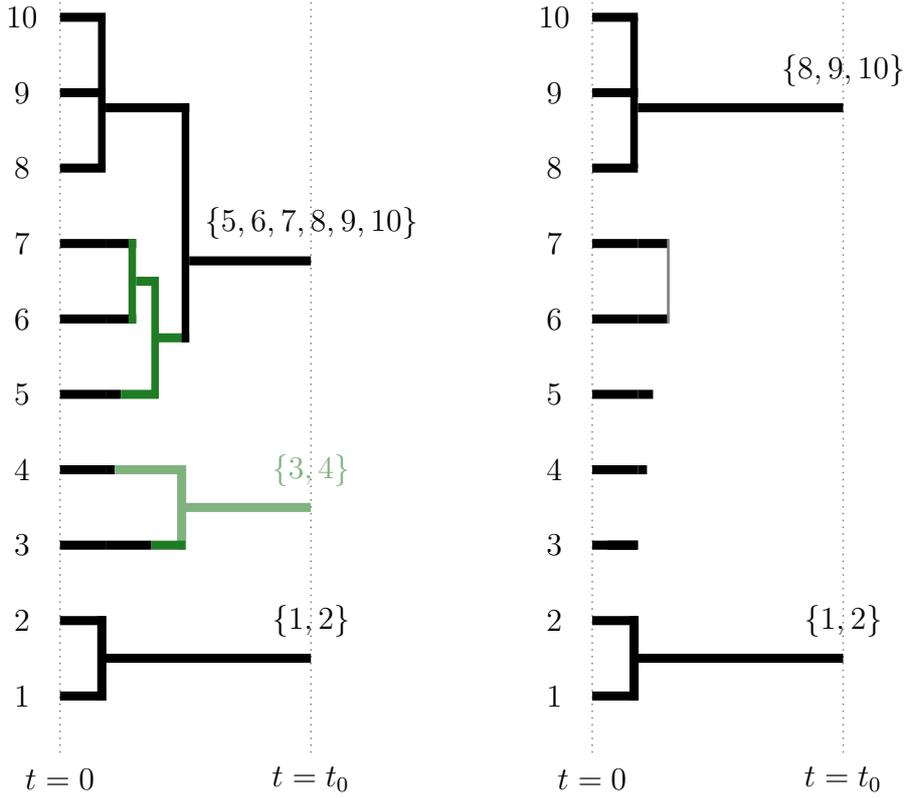

The main idea here is that $(\mathring{\pi}_k(t))_{t \geq 0}$ is in fact a Markov process in its own filtration, and has the dynamics of a single-type $\Lambda$-coalescent with killing. More specifically, the blocks of $(\mathring{\pi}_k(t))_{t \geq 0}$ merge according to the dynamics of a single-type $(\rho_{kk \to k},\bar{Q}_{\to k})$-coalescent, and are killed according to the following dynamics:
\begin{itemize}
\item Individual blocks are eliminated from the process at rate $\sum_{ i \neq k } \rho_{k \to i}$.
\item Let $W_k$ be the measure on $[0,1]$ given by the pushforward of $\sum_{ i \neq k} Q_{ \to k} (\mathrm{d} s)$ under the projection $s \in [0,1]^d \mapsto s_k \in [0,1]$. More specifically, for Borel subsets $A$ of $[0,1]$,
\begin{align*} W_k(A) := \sum_{ i \neq k} \int_{ [ 0, 1]^d} \mathrm{1} \{ s_k \in A \} Q_{ \to i }( \mathrm{d} s).
\end{align*}
For $u \in [0,1]$, at rate $W_k(\mathrm{d}u)$ a large killing event with parameter $u$ happens for the $\mathring{\pi}_k(t)$ process: namely at this event each block of $\mathring{\pi}_k(t)$ is independently killed with probability $u$, or remains alive with probability $1-u$.  The average killing rate due to this effect is given by 
\begin{align*}
\int_0^1 u W_k(\mathrm{d}u) = \sum_{ i \neq k } \int_{[0,1]^d} s_k Q_{\to i}(\mathrm{d}s).
\end{align*}
\end{itemize}
The rate at which each individual block is removed from the process is given by 
\begin{align} 
r_k := \sum_{i \neq k} \int_{[0,1]^d} s_k Q_{ \to i} ( \mathrm{d} s) + \sum_{ i \neq k } \rho_{k  \to i} < \infty,
\end{align}
where the finiteness of $r_k$ follows from the integrability condition \eqref{eq:integrability}. 

As we observed above, the process $(\mathring{\pi}_k(t))_{t \geq 0}$ comes down from infinity. We would now like to utilise this fact to show that the single-type $(\rho_{kk \to k},\bar{Q}_{\to k})$-coalescent without killing also comes down from infinity.

In this direction we have the following lemma.

\begin{lemma} \label{lem:nocares}
Let $(\kappa(t))_{t \geq 0}$ be a single-type $(\rho,Q)$ coalescent, and let $(\mathring{\kappa}(t))_{t \geq 0}$ be a version of $(\kappa(t))_{t \geq 0}$ in which blocks are killed according to the following dynamics:
\begin{itemize}
\item Individual blocks are eliminated from the process at rate $\delta > 0$.
\item Let $W$ be a measure on $[0,1]$. At rate $W(\mathrm{d}u)$ a large killing event with parameter $u$ happens for the blocks of $(\mathring{\kappa}(t))_{t \geq 0}$, at which each block is independently killed with parameter $u$, and remains alive with probability $1-u$. 
\end{itemize}
Suppose that $r := \delta + \int_0^1 u W(\mathrm{d}u ) < \infty$. Then if $\mathring{\kappa}(t))_{t \geq 0}$ comes down from infinity, so does $(\kappa(t))_{t \geq 0}$, i.e.
\begin{align*}
\mathbb{P} ( \# \mathring{\kappa}(t) < \infty ~\forall t > 0) = 1 \qquad \text{implies} \qquad \mathbb{P} ( \# \kappa(t) < \infty~ \forall t>0 ) =1.
\end{align*} 
\end{lemma}
\begin{proof}
There is a natural coupling between $(\kappa(t))_{t \geq 0}$ and $(\mathring{\kappa}(t))_{t \geq 0}$, so that for each $t$, $\mathring{\kappa}(t)$ is a subcollection of the partition $\kappa(t)$. In this coupling, for each block $\Gamma$ of $\kappa(t)$, the probability that the block is also contained in $\mathring{\kappa}(t)$ is $e^{ - r t}$. 

Let $\mathring{\kappa}(t)|_n$ and $\kappa(t)|_n$ denote the restrictions of partitions to $\{1,\ldots,n\}$, as defined in \eqref{eq:restriction}. Write 
\begin{align*}
A_n(t) := \frac{ \# \mathring{\kappa}(t)|_n }{ \# \kappa(t)|_n }.
\end{align*}
Then the expectation of $A_n(t)$ is $e^{- r t}$, and by applying Markov's inequality to the non-negative random variable $1 - A_n(t)$ it follows that $\mathbb{P}( A_n(t) < 1/2 ) \leq 2(1-e^{-rt})$. In particular, there exists a constant $C_r$ such that for all $0 < t \leq 1$ and all $n \in \mathbb{N}$ we have
\begin{align*}
\mathbb{P}( A_n(t) < 1/2 ) \leq C_r t.
\end{align*}
It follows that for all $n,m \in \mathbb{N}$ we have 
\begin{align*}
\mathbb{P}( \# \kappa(t)|_n > m ) &= \mathbb{P}( \# \kappa(t)|_n > m, \# \mathring{\kappa}(t)|_n > m/2 ) + \mathbb{P} ( \# \kappa(t)|_n > m, \# \mathring{\kappa}(t)|_n \leq m/2)\\
&\leq \mathbb{P}( \# \mathring{\kappa}(t)|_n > m/2) + C_r t.
\end{align*}
Now since $\mathbb{P}( \# \mathring{\kappa}(t) = \infty) = \lim_{m \to \infty} \lim_{n \to \infty} \mathbb{P}( \# \mathring{\kappa}|_n > m ) = 0$, we obtain
\begin{align*}
\mathbb{P}( \# \kappa(t) = \infty ) = \lim_{m \to \infty} \lim_{n \to \infty} \mathbb{P}( \# \kappa(t)|_n > m ) \leq C_r t.
\end{align*} 
Now since $\#\kappa(t)$ is almost-surely non-increasing, it follows that whenever $t_1 < t_2$, $\{ \#\kappa(t_2) = \infty \}$ implies $\{ \# \kappa(t_1) = \infty \}$. 
In particular, for any $t > 0$, it follows that $\mathbb{P}( \exists s  > 0 : \# \kappa(s) = \infty) \leq C_r t$. Thus we have 
\begin{align*}
\mathbb{P}( \exists s  > 0 : \# \kappa(s) = \infty) = 0,
\end{align*}
completing the proof.
\end{proof}
We are now equipped to prove one direction of Theorem \ref{thm:coming down}.

\begin{proof}[Proof of `only if' direction of Theorem \ref{thm:coming down}]
Let $(\bpi(t))_{t \geq 0}$ be a multitype $\Lambda$-coalescent that comes down from infinity. We saw at the beginning of this section that we can construct from $(\bpi(t))_{t \geq 0}$ a stochastic process $(\mathring{\pi}_k(t))_{t \geq 0}$ that has the law of single-type $(\rho_{kk \to k},\bar{Q}_{\to k})$-coalescent with stochastic killing at an average rate $r_k < \infty$. Since $(\bpi(t))_{t \geq 0}$ comes down from infinity, and $\mathring{\pi}_k(t)$ has at most the same number of blocks as the $k^{\text{th}}$ component $\pi_k(t)$ of $\bpi(t)$, it follows that $(\mathring{\pi}_k(t))_{t \geq 0}$ also comes down from infinity.

Since $(\mathring{\pi}_k(t))_{t \geq 0}$ has the law of a single-type $(\rho_{kk \to k}, \bar{Q}_{\to k })$-coalescent with killing at a finite expected rate $r_k < \infty$, and $(\mathring{\pi}_k(t))_{t \geq 0}$ comes down from infinity, it follows from Lemma \ref{lem:nocares} that the single-type $(\rho_{kk \to k}, \bar{Q}_{ \to k })$-coalescent without killing also comes down from infinity. 

We now appeal to the work of Bertoin and Le Gall \cite{bertoin2006stochastic} which we discussed in 
Section 1.3,
which states that if 
\begin{align*}
\psi_k(q) := \frac{ \rho_{kk \to k}}{2} q^2 + \int_0^1 e^{ - qu } - 1 + qu ~ \bar{Q}_{ \to k}( \mathrm{d} u ) 
\end{align*}
is the asymptotic processing speed associated with the $(\rho_{kk \to k}, \bar{Q}_{ \to k })$-coalescent that comes down from infinity, then 
\begin{align*}
\int_{ s}^\infty \frac{\mathrm{d} q }{ \psi_k(q)} < \infty \qquad \text{for all $s >0$}.
\end{align*}  
That completes the proof of one direction of Theorem \ref{thm:coming down}.
\end{proof}

\color{black}

\subsection{Proof of the `if' direction in Theorem \ref{thm:coming down}}

Generally speaking, multi-type $\Lambda$-coalescents exhibit complicated behaviour regarding the number of blocks. Due to blocks changing type as well as cross-type mergers, the number of blocks of a certain type may increase whilst numbers of other types are decreasing. What can be said however is that the \emph{total} number of blocks in a multi-type $\Lambda$-coalescent never increases. 

In our first lemma we characterise the average rate of decrease in the number of blocks in the multitype $\Lambda$-coalescent. Recall that we write $(N_1(t),\ldots,N_d(t))$ for the $\mathbb{Z}_{ \geq 0}^d$-valued process where $N_i(t)$ counts the number of blocks of type $i$ in the process at time $t$. 

\begin{lemma} \label{lem:dec rate}
If at some moment we have $(N_1(t),\ldots,N_d(t)) = (n_1,\ldots,n_d)$, then the expected change in the total number of blocks when there are $\mathbf{n}$ blocks is given by 
\begin{align*}
\lim_{ h \downarrow 0} h^{-1} \mathbb{E}[ \sum_{ j=1}^d (N_j(t+h) - N_j(t)) | N(t) = n ] =: - \Psi(n_1,\ldots,n_d),
\end{align*}
where $\Psi:\mathbb{R}^d \to [0,\infty)$ is given by 
\begin{align} \label{eq:Psi}
\Psi(x_1,\ldots,x_d) = \sum_{ i = 1}^d \rho_{ii \to i} \frac{ x_i(x_i-1)}{2}  + \sum_{ i = 1}^d \int_{[0,1]^d} \left\{ \sum_{ j = 1}^d x_j s_j - 1 + \prod_{ j = 1}^d ( 1 - s_j)^{x_j} \right\} Q_{\to i}(\mathrm{d}s).
\end{align}
\end{lemma}

\begin{proof}
We observe that single blocks changing colour have no effect on the total number of blocks. Hence it is only pairwise mergers and larger mergers which cause decreases in the total number of blocks. The former term $\sum_{ i = 1}^d \rho_{ii \to i} \binom{n_i}{2}$ in the definition for $\Psi$ is the expected size of the decrease in the number of blocks due to pairwise mergers. It remains to account for the sum $\sum_{ i = 1}^d R_i(\mathbf{n})$ of the quantities
\begin{align} \label{eq:R def}
R_i(\mathbf{n}) := \int_{[0,1]^d} \left\{ \sum_{ j = 1}^d n_j s_j - 1 + \prod_{ j = 1}^d ( 1 - s_j)^{n_j} \right\} Q_{\to i}(\mathrm{d}s).
\end{align}
To this end, suppose $\mathbf{m} = (m_1,\ldots,m_d)$ blocks are involved in a merger to form a single block of type $i$. Then at this merger the number of blocks of type $j$ decreases by $m_j$ for all $j \neq i$, and the number of type $i$ blocks decreases by $m_i - \mathrm{1} \{ \exists j : m_j > 0 \}$. In particular, the total number of blocks in the process at this merger 
decreases by 
\begin{align*}
m_1 + \ldots + m_d - \mathrm{1} \{ \exists j : m_j > 0 \},
\end{align*}
where $\mathrm{1} \{ \exists j : m_j > 0 \}$ is the indicator function that at least one of the $m_i$ is nonzero.

Now given a type $i$ merger associated with a value $s \in [0,1]^d$, when the process has $\mathbf{n}$ blocks the number of type $j$ blocks involved in the merger is binomially distributed with parameters $n_j$ and $s_j$. In particular, the result follows by noting that if $(M_1,\ldots,M_d)$ is a vector of independent binomially distributed random variables such that each $M_j$ has parameters $(n_j,s_j)$, then 
\begin{align*}
\mathbb{E} \left[  M_1 + \ldots + M_d - \mathrm{1} \{ \exists j : M_j > 0 \} \right] &= \mathbb{E} \left[  M_1 + \ldots + M_d - 1 + \mathrm{1} \{ M_j = 0 ~\forall j \} \right] \\
&=  \sum_{ j = 1}^d n_j s_j - 1 + \prod_{ j = 1}^d ( 1 - s_j)^{n_j} .
\end{align*}
\end{proof}

With a view to controlling from below the rate of decrease in the number of blocks when there are $n$ total blocks, we define the function $\Omega:[0,\infty) \to [0,\infty)$
\begin{align*}
\Omega(x) := \min_{ x_1 + \ldots + x_d = x } \Psi(x_1,\ldots,x_d),
\end{align*}
where $\Psi$ is as in Lemma \ref{lem:dec rate} and the minimum is taken over all non-negative $d$-tuples summing to $x$. 

Our next lemma will allow us to use Jensen's inequality down the line.

\begin{lemma} \label{lem:convex}
The functions $\Psi:\mathbb{R}_{ \geq 0}^d \to \mathbb{R}_{ \geq 0}$ and $\Omega:\mathbb{R}_{ \geq 0} \to \mathbb{R}_{ \geq 0}$ are convex.
\end{lemma}

\begin{proof}
We consider $\Psi$ first. We begin by noting that for each $s$ in $[0,1]^d$, the function 
\begin{align*}
x \mapsto \sum_{ j = 1}^d x_j s _j  - 1 + \prod_{ j = 1}^d (1-s_j )^{x_j} 
\end{align*}
is a convex function from $\mathbb{R}_{ \geq 0}^d$ to $\mathbb{R}_{ \geq 0}$, and hence each integral $R_i(x_1,\ldots,x_d)$ as in \eqref{eq:R def} is also convex. The convexity of $\Psi$ now follows from the fact that the sum of convex functions is convex. 

Turning our attention to $\Omega$, we note from explicit calculation that for real $x,y$ and $\lambda \in [0,1]$ we have 
\begin{align*}
\Omega(\lambda x + (1- \lambda)y) &:= \min_{ z_1 + \ldots + z_d = \lambda x + (1 - \lambda) y } \Psi(z_1,\ldots,z_d)\\
& = \min_{ x_1 + \ldots + x_d = x, y_1 + \ldots + y_d = y } \Psi \left( \lambda x_1 + (1 - \lambda) y_1,\ldots, \lambda x_d + (1 - \lambda) y _d \right) \\
&\leq \min_{ x_1 + \ldots + x_d = x, y_1 + \ldots + y_d = y } \left\{ \lambda \Psi(x_1,\ldots,x_d) + (1 - \lambda) \Psi(y_1,\ldots,y_d) \right \}\\
&= \lambda \Omega(x) + (1 - \lambda) \Omega(y),
\end{align*}
which establishes the convexity of $\Omega$. 

\end{proof}

Recall in the introduction (equation \eqref{eq:proj processing}) we defined the projected type $i$ processing speed $\psi_i : \mathbb{R}_{ \geq 0} \to \mathbb{R}_{ \geq 0}$ to be the function
\begin{align*}
\psi_i(x) := \rho_{ ii \to i } \frac{x(x-1)}{2}+ \int_{[0,1]^d} xs_i - 1 + e^{ - x s_i}  Q_{\to i }( \mathrm{d} s ).
\end{align*}
We will also require the non-asymptotic variant 
\begin{align*}
\tilde{\psi}_i(x) := \rho_{ ii \to i } \frac{x(x-1)}{2}+ \int_{[0,1]^d} xs_i - 1 + (1-s_i)^{x}   Q_{\to i }( \mathrm{d} s ).
\end{align*}
The following lemma states that integrability conditions for $\psi_i$ and $\tilde{\psi}_i$ are equivalent.

\begin{lemma} \label{lem:both finite}
We have
\begin{align*}
\int_s^\infty \frac{\mathrm{d}q}{ \psi_i(q) } < \infty \iff \int_s^\infty \frac{\mathrm{d}q}{ \tilde{\psi}_i(q) } < \infty 
\end{align*}
\end{lemma}
\begin{proof}
Since $(1-s_i)^{q} \leq e^{ - q s_i }$ for all $s_i \in [0,1], q> 0$ we clearly have $\psi_i(q) \geq \tilde{\psi}_i(q)$, so that 
\begin{align*}
 \int_s^\infty \frac{\mathrm{d}q}{ \tilde{\psi}_i(q) } < \infty \implies \int_s^\infty \frac{\mathrm{d}q}{ \psi_i(q) } < \infty .
\end{align*}
On the other hand, we now claim that for every $q \geq 2$, we have $2 \tilde{\psi}_i(q) \geq \psi_i(q)$, which is sufficient to prove the converse claim. Indeed, since $(1 - u/q)^q$ is an increasing function of $q$, setting $ u = qs$ for all $q \geq 2$ we have 
\begin{align*}
2 \left( (1-s)^q - 1 + qs \right) - \left( e^{ - qs } -1  + qs \right) &= 2 \left( (1 - u/q)^q - 1 + u \right) - \left( e^{ - u} - 1 + u \right) \\
&\geq  2 \left( (1 - u/2)^2 - 1 + u \right) - \left( e^{ - u} - 1 + u \right)\\
&\geq 2 \frac{u^2}{4} - u^2/2 = 0,
\end{align*} 
where to obtain the final inequality above we used the fact that $e^{-u} - 1 + u \leq u^2/2$. It follows that whenever $q \geq 2$ we have the inequality
\begin{align*}
2 \left( (1-s)^q - 1 + qs \right)  \geq e^{ - qs} -1  + qs.
\end{align*}
Integrating against $Q_{\to i}(  \mathrm{d}s)$ leads to $2 \tilde{\psi}_i(q) \geq \psi_i(q)$, and in particular 
\begin{align*}
\int_s^\infty \frac{\mathrm{d}q}{ \psi_i(q) } < \infty \implies \int_s^\infty \frac{\mathrm{d}q}{ \tilde{\psi}_i(q) } < \infty ,
\end{align*}
completing the proof.
\end{proof}

Our next lemma bounds the total rate of decrease in the number of blocks in terms of the $\psi_i$. 

\begin{lemma} \label{lem:1lower}
We have
\begin{align*}
\Psi(x_1,\ldots,x_d) \geq \sum_{ j = 1}^d \tilde{\psi}_j(x_j).
\end{align*}
\end{lemma}

\begin{proof}
First we note that for each $i$ we have 
\begin{align*}
&\left\{ \sum_{ j = 1}^d x_j s_j - 1 + \prod_{ j = 1}^d ( 1 - s_j)^{x_j} \right\} - \left\{ x_i s_i - 1 + (1 - s_i)^{x_i} \right\}\\
&= \sum_{ j \neq i } x_j s_j  - (1 - s_i)^{x_i} \left\{ 1 - \prod_{ j \neq i} (1 - s_j)^{x_j } \right\}\\
& \geq \sum_{ j \neq i } x_j s_j  - (1 - s_i)^{x_i} \sum_{ j \neq i } x_j s_j\\
& \geq 0.
\end{align*}
In particular we have
\begin{align*}
\int_{[0,1]^d} \left\{ \sum_{ j = 1}^d x_j s_j - 1 + \prod_{ j = 1}^d ( 1 - s_j)^{x_j} \right\} Q_{\to i } (\mathrm{d} s) \geq \int_{[0,1]^d} \left\{ x_i s_i - 1 + (1 - s_i)^{x_i} \right\} Q_{\to i}( \mathrm{d} s).
\end{align*} 
Summing over $i$, we see that $\Psi(x_1,\ldots,x_d) \geq \sum_{ j = 1}^d \tilde{\psi}_j(x_j)$.
\end{proof}

Our next lemma states that if each $\psi_i$ satsfies the Bertoin-Le Gall \cite{bertoin2006stochastic} integrability condition, so does $\Omega$.

\begin{lemma} \label{lem:integrability lemma}
Suppose for each $i = 1,\ldots,d$ and for all $s > 0$ we have $\int_s^\infty \frac{ \mathrm{d} q}{ \psi_i(q) } < \infty$. Then for all $s > 0$ we have $\int_s^\infty \frac{ \mathrm{d} q }{ \Omega(q) } < \infty.
$
\end{lemma}

\begin{proof}
By Lemma \ref{lem:both finite}, $\int_s^\infty \frac{\mathrm{d}q}{\psi_i(q)} < \infty$ for every $s >0$ implies $\int_s^\infty \frac{\mathrm{d}q}{ \tilde{\psi}_i(q) } < \infty $ for every $s > 0$. 

On the other hand, by Lemma \ref{lem:1lower} we have 
\begin{align} \label{eq:orange}
\Omega(x) := \min_{ x_1 + \ldots + x_d=x  } \Psi(x_1,\ldots,x_d) \geq \min_{ x_1 + \ldots + x_d=x } \sum_{ j = 1}^d \tilde{\psi}_j(x_j) .
\end{align}
Now suppose $x_1 + \ldots + x_d = x$ for non-negative $x_i$. Then at least one of the $x_i$ is greater than $x/d$. In particular from \eqref{eq:orange} we have the crude lower bound
\begin{align*}
\Omega(x) \geq \min_{ i = 1,\ldots,d} \tilde{\psi}_i(x/d).
\end{align*}
It follows that
\begin{align*}
\frac{1}{ \Omega(x) } \leq \frac{1}{ \min_{ i = 1,\ldots,d} \tilde{\psi}_i(x/d)}  = \max_{ i = 1,\ldots,d} \frac{1}{ \tilde{\psi}_i(x/d) } \leq \sum_{ i = 1}^d \frac{1}{ \tilde{\psi}_i(x/d) }.
\end{align*}
It follows that
\begin{align*}
\int_s^\infty \frac{ \mathrm{d} q }{ \Omega(q) } \leq \sum_{ i = 1}^d \int_s^\infty \frac{\mathrm{d} q }{ \tilde{\psi}_i(q/d) } = d \sum_{ i = 1}^d \int_{s/d}^\infty \frac{ \mathrm{d}q}{ \tilde{\psi}_i(q) } < \infty.
\end{align*}
\end{proof}
Our next lemma controls the expected number of blocks in terms of the solution to an ordinary differential equation.

\begin{lemma} \label{lem:control}
For $\mathbf{n} \in \mathbb{Z}_{ \geq 0}^d$, let $v_{\mathbf{n}}(t) := \mathbb{E}_{\mathbf{n}} \left(  N_1(t) + \ldots + N_d(t) \right)$ denote the expected number of blocks in a multitype $\Lambda$-coalescent starting with $n_i$ blocks of type $i$. Then
\begin{align*}
v_{\mathbf{n}}(t) \leq w_n(t),
\end{align*} 
where for $n = n_1 + \ldots + n_d$, $w_n(t)$ is the solution to the integral equation
\begin{align*}
w_n(t) = n - \int_0^t \Omega( w_n(s) ) \mathrm{d} s.
\end{align*}
\end{lemma}

\begin{proof}
Using Lemma \ref{lem:dec rate} to obtain the equality below we have
\begin{align*}
v_{\mathbf{n}}(t) &= n - \int_0^t \mathbb{E} \left[ \Psi(N_1(s),\ldots,N_d(s) ) \right] \mathrm{d}s\\
&\leq n - \int_0^t \Psi\left( \mathbb{E} [ N_1(s)] ,\ldots, \mathbb{E}[N_d(s)] \right) \mathrm{d}s\\
&\leq n - \int_0^t \Omega \left( v_\mathbf{n}(s) \right) \mathrm{d} s
\end{align*}
where the second inequality above follows from using Jensen's inequality and the convexity of $\Psi$ (Lemma \ref{lem:convex}), and the third inequality above follows from the definition of $\Omega$. 

The result now follows from an ODE comparison argument. 
\end{proof}

We are now equipped to prove the more difficult direction of Theorem \ref{thm:coming down}. 

\begin{proof}[Proof of `if' direction of Theorem \ref{thm:coming down}]
Suppose that $\int_{s}^\infty \frac{ \mathrm{d} q}{ \psi_i(q)} < \infty$ for each $i$. Then by Lemma \ref{lem:integrability lemma} it follows that $\int_s^\infty \frac{ \mathrm{d} q }{ \Omega(q) } < \infty$. We now note for each $t > 0$ that $w_n(t)$ defined in \eqref{lem:control} may be written as the unique solution to the equation
\begin{align*}
t = \int_{w_n(t)}^\infty \frac{ \mathrm{d} q}{ \Omega(q) }.
\end{align*}
In particular, for each $t > 0$ we may clearly define $w_\infty(t)$ to be the unique solution to the equation 
\begin{align*}
t = \int_{w_n(t)}^\infty \frac{ \mathrm{d} q}{ \Omega(q) }.
\end{align*}
Plainly for each $t$ we have $w_n(t) \leq w_{n+1}(t) \leq ... \leq w_\infty(t)$. 
In particular, by Lemma \ref{lem:control}, for each $t$ and every $\mathbf{n} \in \mathbb{Z}_{ \geq 0}^d$ we have
\begin{align*}
\mathbb{E}_\mathbf{n} \left[ N_1(t) + \ldots + N_d(t) \right] \leq w_\infty(t) < \infty.
\end{align*}
It follows from the monotone convergence theorem that the multitype $\Lambda$-coalescent comes down from infinity. 
\end{proof}

\section*{Acknowledgements}

This research was supported by the EPSRC funded Project EP/S036202/1 \emph{Random fragmentation-coalescence processes out of equilibrium}.

\bibliographystyle{acm}
\bibliography{2022.03.06jkr}

\end{document}